\documentclass[12pt]{article}
\usepackage{graphicx,amssymb,xcolor}
\usepackage[bookmarks=true]{hyperref}
\newcommand{\gendis}{\kappa}
\newcommand{\dsp}{\displaystyle}
\newcommand{\bd}{\begin{displaymath}}
\newcommand{\be}{\begin{equation}}
\newcommand{\ba}{\begin{array}}
\newcommand{\ed}{\end{displaymath}}
\newcommand{\ee}{\end{equation}}
\newcommand{\ea}{\end{array}}
\newcommand{\espace}{\mbox{ }}

\newcommand{\Prob}{{\rm I\hspace{-0.8mm}P}}
\newcommand{\Exp}{{\rm I\hspace{-0.8mm}E}}
\newcommand{\indicator}[1]{{\mbox{\large\bf$1$}}_{#1}}

\newcommand{\sgn}{\mbox{\rm sgn}\,}

\def\N{\mathbb{N}}
\def\Z{\mathbb{Z}}
\def\R{\mathbb{R}}

\newcommand{\eqref}[1]{(\ref{#1})}

\newtheorem{example}{Example}[section]
\newtheorem{theorem}{Theorem}[section]

\newtheorem{assumption}{Assumption}[section]

\newtheorem{proposition}{Proposition}[section]
\newtheorem{definition}{Definition}[section]
\newtheorem{lemma}{Lemma}[section]
\newtheorem{corollary}{Corollary}[section]

\newtheorem{remark}{Remark}[section]

\newenvironment{proof}[2]{\espace\\{\em Proof of #1 \ref{#2}.}}{\hfill\mbox{$\square$}}

\begin{document}
\title{Quenched convergence and strong local equilibrium for asymmetric zero-range process 
with site disorder}
\author{C. Bahadoran$^{a,e}$, T. Mountford$^{b,e}$, K. Ravishankar$^{c,e}$, E. Saada$^{d,e}$}
\date{\today}
\maketitle
$$ \ba{l}
^a\,\mbox{\small Laboratoire de Math\'ematiques Blaise Pascal, Universit\'e Clermont Auvergne,} \\
\quad \mbox{\small 63177 Aubi\`ere, France} \\
\quad \mbox{\small e-mail:
Christophe.Bahadoran@uca.fr}\\
^b\, \mbox{\small Institut de Math\'ematiques, \'Ecole Polytechnique F\'ed\'erale, 
Lausanne, Switzerland
} \\
\quad \mbox{\small e-mail:
thomas.mountford@epfl.ch}\\
^c\, \mbox{\small NYU-ECNU Institute of Mathematical Sciences at NYU Shanghai, }\\
\quad \mbox{\small  3663 Zhongshan Road North, Shanghai, 200062, China. e-mail:
 kr26@nyu.edu}\\
^d\, \mbox{\small CNRS, UMR 8145, MAP5,
Universit\'e Paris Descartes,
Sorbonne Paris Cit\'e, France}\\
\quad \mbox{\small e-mail:
Ellen.Saada@mi.parisdescartes.fr}\\ 
^e\, \mbox{\small Centre Emile Borel, Institut Henri Poincar\'e, 75005 Paris, France} \\ \\
\ea
$$
\begin{abstract}
We study asymmetric zero-range processes on $\Z$ with nearest-neighbour 
jumps and site disorder. The jump rate of particles is an arbitrary but bounded 
nondecreasing function of the number of particles. 
We  prove quenched strong local equilibrium at subcritical  and critical  hydrodynamic 
densities, and dynamic local loss of mass at supercritical hydrodynamic densities. 
Our results do not assume starting from local Gibbs states. As byproducts 
of these results, we prove convergence of the process from given initial configurations 
with an asymptotic density of particles to the left of the origin.
In particular, we relax the weak convexity assumption of \cite{bmrs1,bmrs2} 
for the escape of mass property.
\end{abstract}
\section{Introduction}
The  asymmetric zero-range processs (AZRP) with site disorder, introduced in \cite{ev}, is an interacting particle system 
whose dynamics is determined by a nondecreasing jump rate  function $g:\N\to\N$, a function 
 $\alpha :\Z^d\to\R_+$ (called the environment or disorder),
and a jump distribution $p(.)$ on $\Z^d$, for $d\ge 1$.    
A particle leaves site $x$ at rate $\alpha(x)g[\eta(x)]$, where $\eta(x)$ denotes the 
current number of particles at $x$, and moves to $x+z$, where $z$ is chosen at random 
with distribution $p(.)$. This  model has product invariant measures.
It exhibits a critical density  $\rho_c$  (\cite{fk,bfl})  if the function $g$ is bounded,
$\alpha$ has averaging properties (precisely, if its empirical measure converges to a limit: for instance, if it is a realization of
a spatially ergodic process), plus a proper tail assumption. \\ \\
 In this paper, we consider the one-dimensional nearest-neighbour process, that is
  $d=1$ and $p(1)+p(-1)=1$.
The particular situation, where $g(n)=\max(n,1)$ and $p(.)$ is concentrated on the value $1$, is a well-known model: namely,
a series of $ M/M/1 $ queues in tandem. It is known from standard queuing theory (see e.g. \cite{par}) 
that the product measure whose marginal at site $x$ is  the geometric distribution with parameter 
$1-\lambda/\alpha(x)$, is invariant for this process, provided $\lambda<\alpha(x)$ for all $x$. 
The parameter $\lambda$ is the intensity of the Poisson process of departures from each queue, hence it can be interpreted as the mean current of customers along the system.
The supremum value of $\lambda$ for which the above invariant measure is defined is $c:=\inf_x\alpha(x)$, that is the maximum possible current value along the system.
This corresponds to a critical density above which no
product invariant measure exists.
More generally, \cite{afgl} showed that there were no invariant measures (whether product or not) of supercritical density, which can be interpreted as a phase transition.
The model of $M/M/1$ queues in tandem has specific properties which make its analysis  more tractable  than that of the general model. For instance, it can be mapped onto last-passage percolation, and the evolution of the system on $(-\infty,x]$ for $x\in\Z$ only depends on the restriction of the initial configuration to this interval.
\\ \\
In the companion paper \cite{bmrs3} and in previous works \cite{bmrs1,bmrs2}, we started developing robust approaches to study 
various aspects of this phase transition for the general model in dimension one with nearest-neighbour jumps (here by ``general'' we mean that we consider a general function $g$ and not only the particular one $g(n)=\max(n,1)$, but do not refer to the dimension or jump kernel).
For instance, an interesting signature of the above phase transition is the {\em mass escape} phenomenon.
Suppose the process is started from 
a given configuration where the global empirical density of customers is greater than $\rho_c$. 
Since the system is conservative, and thus has a whole family of invariant measures 
carrying different mean densities, it is usually expected to converge 
to the extremal invariant measure carrying the same density as the initial state. 
However, in this case such a measure does not exist. For the totally asymmetric 
$M/M/1$ model, 
it was shown in \cite{afgl} 
that the system converges to the maximal invariant measure (thereby implying a loss of mass). This was established 
in \cite{bmrs1,bmrs2} for the general nearest-neighbour model under a weak convexity assumption, and we showed that
this could fail for non nearest-neighbour jump kernels.\\ \\
Phase transition also arises  in the hydrodynamic limit. We showed in \cite{bmrs3} that
the hydrodynamic behavior of our process is given under hyperbolic time scaling by entropy solutions of a scalar conservation law
\be\label{burgers_intro}
\partial_t\rho(t,x)+\partial_x[f(\rho(t,x))]=0
\ee
where $\rho(t,x)$ is the local particle density field, with a flux-density function $\rho\mapsto f(\rho)$ that is increasing up to critical density and constant thereafter.
 In the absence of disorder,  a hydrodynamic limit of the type \eqref{burgers_intro}
was proved (\cite{rez}) in any dimension  for a class of models including the asymmetric zero-range process. In particular, for M/M/1 queues in tandem, 
\be\label{in_which_case}f(\rho)=\rho/(1+\rho)\ee
 In \cite{lan2}, the hydrodynamic limit of the totally asymmetric 
 $M/M/1$ model with a single slow site is shown to be \eqref{burgers_intro} with a special boundary condition, which may induce condensation at the slow site. 
Hydrodynamic limit was then established in \cite{bfl} for the  asymmetric zero-range process with i.i.d. site disorder  in any dimension,
but for subcritical initial data.  
Hydrodynamic limit for the  totally asymmetric 
$M/M/1$ model  with i.i.d. site disorder,  including phase transition,  was proved in \cite{ks},
using the aforementioned mapping with last-passage percolation. The flux function of the disordered model is different from (and not proportional to)  the one in \eqref{in_which_case} and results from a homogenization effect depending jointly  on the disorder distribution {\em and} the jump rate function $g$.\\ \\
The natural question following hydrodynamic limit is that of {\em local equilibrium}. 
In general, for a conservative particle system endowed with a family 
$(\nu_\rho)_{\rho\in\mathcal R}$ of extremal invariant measures (where $\mathcal R$ 
denotes the set of allowed macroscopic densities), the local equilibrium property 
states that the distribution of the microscopic particle configurations around a site 
with macroscopic location $x\in\R^d$ is close to $\nu_{\rho(t,x)}$, where $\rho(t,x)$ 
is the hydrodynamic density, here given by \eqref{burgers_intro}. This property has a 
weak (space-averaged) and a strong (pointwise) formulation, see e.g. \cite{kl}.
In the usual setting of \cite{kl}, provided moment bounds are available, either formulation
actually implies the hydrodynamic limit.
It was shown  in  \cite{lan}  that,  assuming no disorder and strict convexity or concavity of 
the flux function $f$,  
the strong version of local equilibrium 
could be derived in a  fairly 
general  way  from the weak version using monotonicity and translation invariance 
of the dynamics.  The approach of  \cite{lan}  
 requires starting under a local equilibrium product measure. 
 Thus it investigates {\em conservation}, but not {\em spontaneous creation} 
 of local equilibrium.
The latter stronger property is however necessary to study convergence of the process 
from a deterministic initial state (we will explain this point in greater detail below).
Furthermore, the approach of \cite{lan} breaks down 
in the disordered case or (even without disorder) in the absence of strict convexity 
of the flux function. Let us briefly explain these problems.\\ \\
First, in the quenched setting, we lose the translation invariance property used in \cite{lan}. 
Even more fundamentally, the weak and strong local equilibrium measures
 are {\em different} in the  disordered  case. Indeed, hidden in the spatial 
 averaging is a simultaneous averaging on the disorder which would lead to an {\em annealed} 
 equilibrium measure for the weak local equilibrium. Therefore, the very idea of deriving 
 one from the other should be given up. Next, the local equilibrium property is expectedly 
 {\em wrong} at supercritical hydrodynamic densities $\rho(t,x)>\rho_c$, since a corresponding 
 equilibrum measure does not exist in this case. This already poses a problem at the level of the hydrodynamic limit (\cite{bmrs3}),
since the usual heuristic for \eqref{burgers_intro} is based on the idea that the macroscopic flux function $f$ is the expectation of  microscopic flux function under local equilibrium. \\ \\
The purpose of this paper is 
 to introduce a new approach for the derivation of {\em quenched strong} local equilibrium 
 in order to address this question, which was left open by the previous works \cite{bfl,sep}. 
In the case of supercritical hydrodynamic density $\rho(t,x)>\rho_c$, we prove that 
the local equilibrium property fails, and that, locally around ``typical points''  of the environment, 
the distribution of the microscopic state is close to the {\rm critical measure} with density 
$\rho_c$: this can be viewed as a {\sl dynamic version} of the loss of mass property 
studied in \cite{afgl,bmrs1,bmrs2}. The dynamic loss of mass that we establish here allows 
us to remove the convexity assumption  used  in \cite{bmrs1,bmrs2}, but for a slightly less general 
class of initial configurations.\\ \\
Our approach to strong local equilibrium yields new results even in the special case of 
homogeneous systems, that is $\alpha(x)\equiv 1$. Indeed, 
in this paper, we  establish  {\em spontaneous creation} of local equilibrium: 
that is, we only require starting from a sequence of (possibly deterministic) initial 
configurations with a given macroscopic profile, but with a distribution far away from 
local equilibrium. 
Thanks to this, we obtain fairly complete convergence results for the process starting 
from a {\em given} initial state.
Such results are usually difficult to obtain for asymmetric models, where most convergence 
results start from an initial spatially ergodic distribution. \\ \\
The connection alluded to above between creation of local equilibrium and convergence results can be understood by
letting  the system start initially from a uniform hydrodynamic profile with density $\rho$. The result of \cite{bmrs3} implies that no evolution is seen at all on the hydrodynmic scale, because uniform profiles are stationary solutions of the hydrodynamic equation.
If $\rho<\rho_c$, one way to achieve this profile is  by distributing the initial configuration according to the product stationary state of the dynamics with mean density $\rho$. Then
no evolution will be seen microscopically either, and {\em conservation} of local equilibrium reduces to a local
formulation of this stationarity. In contrast, assume the initial flat profile is achieved by a non-stationary distribution (for instance, a deterministic configuration). Then, while the hydrodynamic profile does not evolve in time, the local equilibrium {\em creation} property implies that the out-of-equilibrium configuration converges to the product invariant measure with mean density $\rho$ if $\rho<\rho_c$,
or to the critical invariant measure if $\rho\geq\rho_c$. Thus, behind the {\em macroscopic} stationarity of the system on the hydrodynamic scale,
two different mechanisms of {\em microscopic} evolution are concealed: the {\em convergence} mechanism, by which the system locally sets itself under an equilibrium measure, and the {\em mass escape} mechanism, by which, locally, the supercritical mass concentrates itself around slowest sites, leading to a local equilibrium measure that disagrees with the hydrodynamic density. We refer the reader to Example \ref{ex_ill} in Subsection \ref{subsec:results} for a precise formulation of the above discussion.
\\ \\
Similar local equilibrium and convergence results were obtained 
in \cite{bam} for the one-dimensional ASEP without disorder. However, in the approach used there, 
a key ingredient was the strict concavity of the flux function $f$. This property enabled 
 the authors  to use the {\em a priori} two-block estimate of \cite{kos} as a first step to obtain 
a spatially averaged version of the result. In the present case, we introduce a 
new argument that assumes only genuine nonlinearity of $f$ (that is, the absence of any 
linear portion on its graph).  Under this assumption, we prove the spatially averaged property, 
which actually would imply an a priori two-block estimate, thereby extending the result of 
\cite{kos} to flux functions  without strict convexity or concavity.   \\ \\
The paper is organized as follows. In Section \ref{sec_results}, we introduce the model 
and notation, and state our main results. 
 In Section \ref{sec:preliminary} we recall the graphical construction of the model,
earlier results on currents, and we define specific couplings.  
In Section \ref{sec:proof_loc_eq}, we prove quenched strong local equilibrium 
around subcritical hydrodynamic densities. 
 In Sections \ref{sec:proof_loc_eq_2} and \ref{sec:cesaro}, we prove the dynamic loss of 
mass property around supercritical hydrodynamic densities. Finally, in Section \ref{sec:proof_bfl}, we investigate the problem of convergence starting from a given initial configuration.
\section{Notation and results}
\label{sec_results}
In the sequel,  $\R$ denotes the set of real numbers, 
$\Z$  the set of signed integers, $\N=\{0,1,\ldots\}$ 
the set of nonnegative integers
 and $\overline{\N}:=\N\cup\{+\infty\}$. 
For $x\in\R$,  $\lfloor x\rfloor$  
denotes the integer part of $x$, that is largest integer $n\in\Z$ 
such that $n\leq x$.  
The notation $X\sim\mu$ means that a random 
variable $X$ has probability distribution $\mu$.\\ \\
Let   $\overline{\mathbf{X}}:=\overline{\N}^{\Z}$   denote the set of particle configurations, 
and ${\mathbf{X}}:=\N^\Z$ the subset of particle configurations with finitely many particles at each site.
 A configuration in $\overline{\mathbf{X}}$ is of the form  $\eta=(\eta(x):\,x\in\Z)$ 
 where $\eta(x)\in\overline{\N}$ for each $x\in\Z$. 
 The set  $\overline{\mathbf{X}}$  is equipped with the  coordinatewise  order:  for  
 $\eta,\xi\in\overline{\mathbf{X}}$,  we write 
$\eta\leq\xi$ if and only if $\eta(x)\leq\xi(x)$ for every $x\in\Z$;  in the latter 
inequality, $\leq$ stands for extension to $\overline{\N}$ of
the natural order on $\N$, defined by $n\leq+\infty$ for every $n\in\N$, and $+\infty\leq+\infty$. 
  This order is extended to probability measures on $\overline{\mathbf{X}}$: For two 
 probability measures $\mu,\nu$, we write $\mu\le\nu$ if and only if
$\int fd\mu\le\int fd\nu$ for any nondecreasing function $f$ on $\overline{\mathbf{X}}$.  
We denote by $(\tau_x)_{x\in\Z}$ the group of spatial shifts. For $x\in\Z$, the action of $\tau_x$  on particle configuration 
is defined by $(\tau_x\eta)(y)=\eta(x+y)$ for every $\eta\in\overline{\bf X}$. Its action on a function $f$ from $\bf X$ or
$\overline{\bf X}$ to $\R$ is defined by $\tau_x f:=f\circ\tau_x$.
\subsection{The process and its invariant measures}\label{subsec:procinv}
Let  $p(.)$ be a probability measure on $\Z$ supported on $\{-1,1\}$.
 We set $p:=p(1)$, $q=p(-1)=1-p$, and assume $p\in(1/2,1]$,
so that the mean drift of the associated random walk is $p-q>0$.
\\
Let $g:\N\to[0,+\infty)$ be a nondecreasing function such that
\be\label{properties_g}
g(0)=0<g(1)\leq \lim_{n\to+\infty}g(n)=:g_\infty<+\infty
\ee
We extend $g$ to  $\overline{\N}$  by setting $g(+\infty)=g_\infty$.
Without loss of generality, we henceforth assume 
\be\label{ginfty}
g(+\infty)=g_\infty=1\ee   
Let 
$\alpha=(\alpha(x),\,x\in\Z)$ (called the environment or disorder) be a 
 $[0,1]$-valued sequence. 
The set of environments is denoted by
\be\label{set_env}
{\bf A}:=[0,1]^{\Z}
\ee
We consider the 
Markov process  $(\eta_t^\alpha)_{t\geq 0}$
on  $\overline{\mathbf{X}}$  with generator given for any cylinder function  
$f:\overline{\mathbf{X}}\to\R$  by
\be\label{generator}
L^\alpha f(\eta)  =  \sum_{x,y\in\Z}\alpha(x)
p(y-x)g(\eta(x))\left[
f\left(\eta^{x,y}\right)-f(\eta)
\right]\ee
where, if $\eta(x)>0$,  $\eta^{x,y}:=\eta-\mathfrak{\delta}_x+\mathfrak{\delta}_y$ 
denotes the new configuration obtained from $\eta$ after a particle has 
jumped from $x$ to $y$ (configuration  $\mathfrak{\delta}_x$ has one particle at $x$ and 
no particle elsewhere; addition of configurations is meant coordinatewise). 
 In cases of infinite particle number,  the following interpretations hold:
  $\eta^{x,y}=\eta-\mathfrak{\delta}_x$ if $\eta(x)<\eta(y)=+\infty$ (a particle is removed from $x$), 
  $\eta^{x,y}=\eta+\mathfrak{\delta}_y$ if $\eta(x)=+\infty>\eta(y)$ (a particle is created at $y$),
$\eta^{x,y}=\eta$ if $\eta(x)=\eta(y)=+\infty$. \\ \\
This process has the property 
that if $\eta_0\in{\mathbf{X}}$, then almost surely, one has $\eta_t\in{\mathbf{X}}$ 
for every $t>0$. In this case, it may be considered as a Markov process on $\mathbf{X}$ with generator
\eqref{generator} restricted to functions $f:{\mathbf{X}}\to\R$. \\
 When the environment $\alpha(.)$ 
is identically equal to $1$, we recover
the {\em homogeneous} zero-range process (see \cite{and} for its detailed analysis). \\ \\ 
For the existence and uniqueness of  $(\eta_t^\alpha)_{t\geq 0}$  
see \cite[Appendix B]{bmrs2}.   
 Recall from \cite{and} that, since $g$ is nondecreasing, 
 $(\eta_t^\alpha)_{t\geq 0}$  
is attractive, i.e.
its semigroup maps nondecreasing functions (with respect to 
the partial order on $\overline{\mathbf{X}}$) onto nondecreasing functions.
One way to see this is to construct a monotone coupling of  two copies of the process, 
see Subsection \ref{subsec:Harris} below. \\ \\
For ${\beta}<1$,  
we define the probability measure $\theta_{{\beta}}$ on $\N$
 by \label{properties_a}
\begin{equation}\label{eq:theta-lambda}
\theta_{{\beta}}(n):=Z({\beta})^{-1}\frac{\beta^n}{g(n)!},\quad n\in\N,
\qquad\mbox{where}\quad Z(\beta):=\sum_{\ell=0}^{+\infty}\frac{\beta^\ell}{g(\ell)!}
\end{equation}
We denote by  $\mu_\beta^\alpha$  the invariant measure 
of $L^\alpha$ defined (see e.g. \cite{bfl}) as the product measure 
with marginal $\theta_{\beta/\alpha(x)}$ at site $x$: 
\be\label{def_mu_lambda_alpha}
\mu^\alpha_{\beta}(d\eta):=\bigotimes_{x\in\Z}\theta_{\beta/\alpha(x)}[d\eta(x)]
\ee
 The measure \eqref{def_mu_lambda_alpha} can be defined on $\overline{\mathbf{X}}$  for
\be\label{cond_simple_bar}
{\beta}\in[0,\inf_{x\in\Z}\alpha(x)]
\ee
by using the conventions
\be\label{extension_theta}\theta_1:=\delta_{+\infty}
\ee
\be
\label{convention_2}
\frac{\beta}{a}=0\mbox{ if }\beta=0\mbox{ and }a\geq 0
\ee
The measure \eqref{def_mu_lambda_alpha} is always supported on $\bf X$ if
\be\label{cond_lambda_simple}
\beta\in(0,\inf_{x\in\Z}\alpha(x))\cup\{0\}
\ee
When  $\beta=\inf_{x\in\Z}\alpha(x)>0$, conventions  \eqref{extension_theta}--\eqref{convention_2} 
yield  a measure supported on configurations with infinitely many particles 
at all sites $x\in\Z$ that achieve $\inf_{x\in\Z}\alpha(x)$, and finitely many particles at other sites.
In particular, this measure is supported on $\bf X$  when $\inf_{x\in\Z}\alpha(x)$ is not achieved. When $\inf_{x\in\Z}\alpha(x)=0$, the measure \eqref{def_mu_lambda_alpha} is supported on the empty configuration.
Since 
$(\theta_{{\beta}})_{{\beta}\in [0,1)}$ is an exponential family, we have that,  for $\beta\in[0,\inf_{x\in\Z}\alpha(x)]$, 
 \be\label{stoch_inc}
 \mu^\alpha_{{\beta}}\,
 \mbox{ is weakly continuous and stochastically increasing with respect to }{\beta}
 \ee
\subsection{ The hydrodynamic limit}
\label{subsec:hydro}
Let us first recall the results of \cite{bmrs3} with respect to the hydrodynamic limit of our process. We begin with the 
following standard 
definitions in hydrodynamic limit theory. We denote by
$\mathcal M(\R)$ the set of Radon measures on $\R$. To a particle configuration
 $\eta\in{\mathbf X}$, we associate a sequence of empirical measures $(\pi^N(\eta):\,N\in\N\setminus\{0\})$ defined by
$$
\pi^N(\eta):=\frac{1}{N}\sum_{y\in\Z}\eta(y)\delta_{y/N}\in\mathcal M(\R)
$$
Let $\rho_0(.)\in L^\infty(\R)$, and let $(\eta^N_0)_{N\in\N\setminus\{0\}}$ denote 
a sequence of $\bf X$-valued random variables. We say this sequence has 
limiting density profile $\rho_0(.)$, if the sequence of empirical measures 
$\pi^N(\eta^N_0)$ converges in probability to the deterministic measure 
$\rho_0(.)dx$ with respect to the topology of vague convergence.    \\ \\
\textbf{Assumptions on the environment.} To state the results of \cite{bmrs3},  we introduce the following assumptions on $\alpha$.
\begin{assumption}\label{assumption_ergo}
There exists a probability measure $Q_0=Q_0(\alpha)$
on $[0,1]$ such that
\be\label{eq:assumption_ergo}
Q_0(\alpha)=\lim_{n\to+\infty}\frac{1}{n+1}\sum_{x=-n}^0 \delta_{\alpha(x)}
=\lim_{n\to+\infty}\frac{1}{n+1}\sum_{x=0}^n \delta_{\alpha(x)}
\ee
\end{assumption}
Assumption \ref{assumption_ergo} means that the environment $\alpha$ has ``averaging properties''.
Note that it implies
\be\label{support_Q}
{ c:=}\inf_{x\in\Z}\alpha(x)\leq\inf{\rm supp}\,Q_0
\ee
The next assumption sets a restriction
on the sparsity of  slow sites (where by ``slow sites''  
we mean sites where the disorder variable becomes arbitrarily close 
or equal to the infimum value $c$). 
\begin{assumption}\label{assumption_dense} We say that the environment $\alpha$ 
has  {\em macroscopically dense defects} if 
there exists a sequence  of sites  $(x_n)_{n\in\Z}$  such that 
\be\label{assumption_afgl}
\forall n\in\Z,\,x_n<x_{n+1};\quad
\lim_{n\to\pm\infty}\alpha(x_n)= c 
\ee
and
\be\label{assumption_afgl_0}
\lim_{n\to\pm\infty}\frac{x_{n+1}}{x_n}=1
\ee
\end{assumption}
Remark that \eqref{assumption_afgl} implies in particular
\be\label{not_too_sparse}
\liminf_{x\to\pm\infty}\alpha(x)= c
\ee
The role of Assumption \ref{assumption_dense}, and consequences of its violation, are discussed in detail  in \cite[Section 3]{bmrs3}.
Assumptions \ref{assumption_ergo} and \ref{assumption_dense} are satisfied for instance in the case of an ergodic random environment:
\begin{example}\label{example_random}
Let $Q$ be a spatially ergodic probability measure on $\bf A$ with marginal $Q_0$
(for instance,  we may consider the i.i.d. case  $Q=Q_0^{\otimes\Z}$). Then, $Q$-almost every $\alpha\in\bf A$ satisfies
Assumption \ref{assumption_ergo} and equality in \eqref{support_Q}. 
\end{example}
Remark that in Example \ref{example_random}, $Q$-a.e. realization of the environment $\alpha$ also satisfies Assumption \ref{assumption_dense}. Following is an example of a deterministic environment taken from \cite[Section 3]{bmrs3} satisfying both assumptions \ref{assumption_ergo} and \ref{assumption_dense}. Unlike Example \ref{example_random}, in this example, \eqref{support_Q} may be a strict inequality.
\begin{example}\label{example_det}
Let
$$
F_{Q_0}(u):=Q_0([0,u]),\quad\forall u\geq 0
$$
denote the cumulative distribution function of $Q_0$. Let ${\mathcal X}=\{x_n:\,n\in\Z\}$
be  a  doubly infinite $\Z$-valued sequence, and $(\alpha_n)_{n\in\Z}$ a doubly infinite $[0,1]$-valued sequence, 
such that
\begin{eqnarray*}
\lim_{n\to\pm\infty}\alpha_n & = & c\leq\inf{\rm supp}\,Q_0\\
\lim_{n\to\pm\infty}\frac{x_{n+1}}{x_n} & = & 
1\\
\lim_{n\to\pm\infty}\frac{n}{x_n} & = &
0\\
\end{eqnarray*}
Define
\begin{eqnarray}
u(x) & := & \sum_{n\in\Z}1_{[{x}_n,{x}_{n+1})}(x)\frac{x- { x_n}}{  x_{n+1}-x_n}\label{sample_uniform}\\ \nonumber\\
\alpha(x) & := &  F_{Q_0}^{-1}(u(x)) \label{inversion}
{\bf 1}_{\Z\setminus{\mathcal X}}(x)+{\bf 1}_{\{c<\inf{\rm supp}\,Q_0\}}\sum_{n\in\Z}\alpha_n{\bf 1}_{\{ {x}_n\}}(x)
\end{eqnarray}
For instance,
\be\label{ex_det_ex}
 x_n=
{\bf 1}_{\{n\neq 0\}}\sgn(n)\lfloor{|n|}^\kappa\rfloor,\quad
\kappa>1
\ee
\end{example}
\mbox{}\\
\textbf{A family of flux functions.} Let $\mathcal D=\{(Q_0,\gamma)\in{\mathcal P}([0,1])\times [0,1]:\,\gamma\leq\inf{\rm supp}\,Q_0\}$. We define  a family 
$\{f^{Q_0,\gamma}:\,(Q_0,\gamma)\in{\mathcal D}\}$ of functions $f^{Q_0,\gamma}:[0,+\infty)\to[0,+\infty)$
as follows. 
Let
\be\label{mean_density}
R({\beta}):=\sum_{n=0}^{+\infty}n\theta_{{\beta}}(n),\quad\beta\in[0,1[
\ee
denote the mean value of $\theta_\beta$.
The function $R(.)$ is  analytic on $[0,1)$, increasing from $0$ to $+\infty$,  and extended (cf. \eqref{extension_theta}) by setting
$R(1)=+\infty$.  
We now set
\begin{eqnarray}
\overline{R}^{Q_0}({\beta}) 
& :=  & \int_{[0,1]}R\left[\frac{\beta}{a}\right]dQ_0[a]
\label{average_afgl_ergodic}\\
\label{critical_parameter_max}
\rho_c(Q_0) & := & \overline{R}^{Q_0}[
\inf{\rm supp}\,Q_0
]
\end{eqnarray}
 For $\alpha\in{\bf A}$ satisfying \eqref{eq:assumption_ergo}, and $\beta<\inf_{x\in\Z}\alpha(x)$, the quantity $\overline{R}^{Q_0(\alpha)}(\beta)$ (see Lemma \ref{lemma_inter} below) represents the asymptotic mean density of particles under
$\mu^\alpha_\beta$. 
The function $\overline{R}^{Q_0}$ is increasing and continuous on the interval $[0,\inf{\rm supp}\,Q_0]$
(see \cite[Lemma 3.1]{bmrs3} and Lemma \ref{lemma_properties_flux} below). 
We may thus define the inverse of $\overline{R}^{Q_0}$
on its image, and set
\be\label{flux_parameter}
f^{Q_0,\gamma}(\rho):=\left\{
\ba{lll}
(p-q)\left(\overline{R}^{Q_0}\right)^{-1}(\rho) & \mbox{if} & \rho<
 \overline{R}^{Q_0}(\gamma)\\
(p-q)\gamma & \mbox{if} & \rho\geq \overline{R}^{Q_0}(\gamma)
\ea
\right.
\ee
\be\label{def_f_alpha}
f^\alpha:=f^{Q_0(\alpha),\inf\alpha},\quad \rho_c^\alpha:= \rho_c(Q_0(\alpha))
\ee
Useful Properties of   $\overline{R}^{Q_0}$  and $f^{Q_0,\gamma}$ are stated in the following lemma, which is a consequence of
\cite[Lemma 3.1]{bmrs3}.
\begin{lemma}\label{lemma_properties_flux} 
 (i)  The function $\overline{R}^{Q_0}$  and
its inverse are increasing and analytic, respectively from  $[0,\inf{\rm supp}\,Q_0]$ to $[0,\rho_c(Q_0)]\cap\R$ 
and from $[0,\rho_c(Q_0)]\cap \R$ to $[0,\inf{\rm supp}\,Q_0]$.   (ii) There is a dense subset $\mathcal R$ 
of  $[0,\rho_c(Q_0)]\cap\R$  such that for any $\rho\in\mathcal R\cap[0,\gamma]$, $f^{Q_0,\gamma}$ is either 
uniformly strictly convex or uniformly strictly concave in a neighbourhood 
of $\rho$. 
\end{lemma}
Note that only statement (i) of Lemma \ref{lemma_properties_flux} is contained in Lemma 3.1 of \cite{bmrs3}, but (ii) is a consequence thereof.
We can now state the main result of \cite{bmrs3}.
\begin{theorem}
\label{th_hydro}
(\cite[Theorem 2.1]{bmrs3})
Assume the environment $\alpha$ satisfies  Assumption \ref{assumption_ergo}, 
and the  sequence $(\eta^N_0)_{N\in\N\setminus\{0\}}$ has limiting density profile $\rho_0(.)\in L^\infty(\R)$. 
For each $N\in\N\setminus\{0\}$, let  $(\eta^{\alpha,N}_{t})_{t\geq 0}$  denote the process with initial 
configuration $\eta^N_0$ and generator \eqref{generator}. Assume either that 
the initial data is subcritical, that is $\rho_0(.)< \rho_c^{\alpha}$; or, that the 
defect density Assumption \ref{assumption_dense} holds.
 Let $\rho(.,.)$ denote the entropy solution to 
\be\label{conservation_law}
\partial_t \rho(t,x)+\partial_x f^\alpha[\rho(t,x)]=0
\ee
 with initial datum $\rho_0(.)$.
Then for any $t>0$,
 the sequence  $(\eta^{\alpha,N}_{Nt})_{N\in\N\setminus\{0\}}$ 
 has limiting density profile $\rho(t,.)$.
\end{theorem}
\subsection{Results}\label{subsec:results}
 We  now come to  the main results of this paper, which are  concerned with 
strong local equilibrium with respect to the hydrodynamic limit \eqref{conservation_law}. 
To state these results, we need a parametrization of the invariant measures as a function of the asymptotic density.
By Lemma \ref{lemma_properties_flux} and \eqref{def_f_alpha},  we may define
\be\label{invariant_density}
\mu^{\alpha,\rho}:=\mu^\alpha_{(\overline{R}^{Q_0(\alpha)})^{-1}(\rho)},\,
\theta^{\alpha,\rho}:=\theta^\alpha_{(\overline{R}^{Q_0(\alpha)})^{-1}(\rho)},\,
\quad\forall \rho\in[0,\rho_c^\alpha)
\ee
The following lemma is established in \cite{bmrs3}.
\begin{lemma}(\cite[Lemma 2.1]{bmrs3})
\label{lemma_inter}
Let 
$\rho\in[0,\rho_c^\alpha)$, and
$\eta^{\alpha,\rho}\sim\mu^{\alpha,\rho}$.
Then the following limit holds in probability:
that is, 
\be\label{justify_mean}
\lim_{n\to+\infty}\frac{1}{n+1}\sum_{x=0}^n\eta^{\alpha,\rho}(x)=
\lim_{n\to+\infty}\frac{1}{n+1}\sum_{x=0}^n\eta^{\alpha,\rho}(-x)=\rho
\ee
\end{lemma}
\begin{remark}\label{remark_failure}
We did not extend the reindexation \eqref{invariant_density} up to $\rho=\rho_c^\alpha$. Indeed, in so doing, we would obtain
$\mu^{\alpha,\rho}:=\mu^\alpha_c$ for $\rho=\rho_c^{\alpha}$. With this definition, the limit \eqref{justify_mean} 
 may fail (see \cite[Section 2.2]{bmrs3} for details). 
\end{remark}
 In the sequel, we will simplify notation by setting $\rho_c=\rho_c^\alpha$, $\overline{R}=\overline{R}^{Q_0(\alpha)}$ and
$f=f^\alpha$. Let 
\be\label{limsupinf}
\rho^*(t,u):=\limsup_{(t',u')\to(t,u)}\rho(t',u'),\quad
\rho_*(t,u):=\liminf_{(t',u')\to(t,u)}\rho(t',u')
\ee
The  
results will be different whether one considers a point around which
the local hydrodynamic density is  subcritical, critical, or supercritical.   In the  first  case, a corresponding 
invariant measure exists, and we show in the following theorem that the microscopic distribution locally approaches this measure. 
\begin{theorem}\label{th_strong_loc_eq}
Under assumptions of Theorem \ref{th_hydro}, the following holds for every $(t,u)\in(0,+\infty)\times\R$:
let $\psi:{\mathbf X}\to\R$  
be a bounded local function, and $(x_N)_{N\in\N}$ a sequence  of sites 
such that  $u=\lim_{N\to+\infty}N^{-1}x_N$. Then:\\ \\
 (i) If $\psi$ is nondecreasing and $\rho_*(t,u)<\rho_c$, 
\be\label{eq:loc_eq_super_lower}
\lim_{N\to+\infty}\left[
\Exp 
\psi\left(
\tau_{x_N}\eta^{\alpha,N}_{Nt}
\right)-\int_{\mathbf X}\psi(\eta)d\mu^{
\tau_{x_N}\alpha,\rho_*(t,u)
}(\eta)
\right]^-=0
\ee
 (ii) If $\psi$ is nondecreasing
and $\rho^*(t,u)<\rho_c$, 
\be\label{eq:loc_eq_upper}
\lim_{N\to+\infty}\left[
\Exp 
\psi\left(
\tau_{x_N}\eta^{\alpha,N}_{Nt}
\right)-\int_{\mathbf X}\psi (\eta)d\mu^{
\tau_{x_N}\alpha,\rho^*(t,u)
}(\eta)
\right]^+=0
\ee
 (iii) If $\rho(.,.)$ is continuous at $(t,u)$ and $\rho(t,u)<\rho_c$, 
\be\label{eq:loc_eq}
\lim_{N\to+\infty}\left[
\Exp 
\psi\left(
\tau_{x_N}\eta^{\alpha,N}_{Nt}
\right)-\int_{\mathbf X}\psi (\eta)d\mu^{
\tau_{x_N}\alpha,\rho(t,u)
}(\eta)
\right]=0
\ee
\end{theorem}
Statement \eqref{eq:loc_eq} differs from the usual strong local equilibrium 
statement in the sense that there is no limiting measure: rather, 
the microscopic distribution around some macroscopic point is close to 
an equilibrium distribution which itself varies along the disorder.\\ \\
The analogue of \cite{lan} in our context, that is {\em conservation} of local equilibrium,  corresponds to the following particular case of Theorem \ref{th_strong_loc_eq}.
\begin{example}\label{ex_cons}
Assume $\rho_0(.)$ takes values in $[0,\rho_c)$, and $\eta^N_0\sim\mu^N_0$, where
\be\label{def_init_loceq}
\mu^N_0:=\bigotimes_{x\in\Z}\theta^{\alpha,\rho^N_x},
\ee
and $(\rho^N_x:\,x\in\Z,\,N\in\N\setminus\{0\})$ is a  $[0,\rho_c)$-valued family such that
\be\label{cond_init_loceq}
\lim_{N\to+\infty}\rho^N_{\lfloor N.\rfloor}=\rho_0(.),\quad\mbox{in }L^1_{\rm loc}(\R)
\ee
\end{example}
The difference with {\em creation of local equilibrium} is that the latter result only assumes that the sequence $(\eta^N_0)_{N\geq 1}$
has a limiting density profile as in Theorem \ref{th_hydro}.
Previous works \cite{av,lan} established \eqref{eq:loc_eq} starting from random initial configurations of the form \eqref{def_init_loceq}--\eqref{cond_init_loceq} for the {\em homogeneous} asymmetric zero-range process, that corresponds to the uniform environment $\alpha(x)\equiv 1$.
In this case, $\rho_c=+\infty$. Both works require the additional assumption that the flux function $f$ is strictly concave. In \cite{av}, conservation of local equilibrium is established directly for initial data of the form
\be\label{riemann_data}
\rho_0(x)=\rho_l{\bf 1}_{(-\infty,0)}(x)+\rho_r{\bf 1}_{(0,+\infty)}(x)
\ee
that is the so-called {\em Riemann problem} for \eqref{conservation_law}.
In \cite{lan}, the result is established for  general initial data in any dimension. The approach of \cite{lan} consists in deriving conservation of local equilibrium in the strong sense \eqref{eq:loc_eq} from conservation of local equilibrium in the {\em weak sense} previously established (together with the hydrodynamic limit \eqref{conservation_law}) in \cite{rez}. In \cite{bfl}, the hydrodynamic limit was derived in any dimension for the asymmetric  zero-range process with i.i.d. site disorder, starting from subcritical initial distributions of the form \eqref{cond_init_loceq}, but the question of local equilibrium was left open. The problem of convergence from given initial configurations was formulated in \cite{bfl} and is addressed in Theorem \ref{cor_dream_bfl} using our local equilibrium results (Theorems \ref{th_strong_loc_eq} above and \ref{th_strong_loc_eq_super}) as the main ingredient. \\ \\
Next, we consider points at which the hydrodynamic density is supercritical 
 or critical. In the first case,  there does not exist a corresponding equilibrium measure.
Therefore one cannot expect the same type of convergence as above. We show that in this case, 
one has local convergence to the {\em critical} quenched invariant  measure. The nature of 
this phenomenon is different from the result of Theorem \ref{th_strong_loc_eq}, and motivates 
an additional assumption that we now introduce. 
\begin{definition}\label{def_typical}
Let $(x_N)_{N\in\N}$ be a sequence of sites such that $N^{-1}x_N$ converges to $u\in\R$ as $N\to+\infty$.  
The sequence $(x_N)_{N\in\N}$ is {\em typical} if and only if 
any subsequential limit $\overline{\alpha}$ of 
 $(\tau_{x_N}\alpha)_{N\in\N}$ 
has the following properties:\\ \\
(i) For every $z\in\Z$, 
$\overline{\alpha}\in{\bf B}:=(c,1]^\Z$.\\ 
(ii) 
$\liminf_{z\to-\infty}\overline{\alpha}(z)=c$.
\end{definition}
 The interpretation of the word ``typical'' is the following. When $\alpha$ does not achieve its infimum value $c$ (that is $\alpha\in\bf B$),
(i)--(ii) says that the environment as seen from $x_N$ shares  key properties of the environment 
seen from a fixed site (say the origin). Indeed, (i) means that one has to look 
infinitely faraway to see slow sites 
or never sees them, and (ii) means
that one eventually does encounter such sites.  Of course, when $\alpha$ does achieve its infimum, assumption (i) 
is no longer true for the environment seen from the origin. However, we point out that assumption (i)  could be removed at the expense of a slightly longer proof, that we omit  for simplicity.  

The typicality assumption will in fact be needed only to show that the  
microscopic distribution is locally dominated by the critical measure.
We observe however that even in the critical case, where the invariant measure does exist, 
we are not able to prove the statement without the typicality assumption. 
\begin{theorem}\label{th_strong_loc_eq_super}
Under assumptions  of Theorem \ref{th_hydro}, 
the following holds for every $(t,u)\in(0,+\infty)\times\R$.
Let $\psi:{\mathbf X}\to\R$ 
be a bounded local function, and $(x_N)_{N\in\N}$ a typical sequence  of sites 
such that  $u=\lim_{N\to+\infty}N^{-1}x_N$. Then: \\ \\
(i) If $\psi$ is nondecreasing, 
\be\label{eq:loc_eq_super}
\lim_{N\to+\infty}\left[
\Exp 
\psi\left(
\tau_{x_N}\eta^{\alpha,N}_{Nt}
\right)-
\int_{\overline{\mathbf X}}\psi(\eta)d\mu_c^{
\tau_{x_N}\alpha
}(\eta)
\right]^+=0
\ee
(ii) If  $\rho_*(t,u)\geq\rho_c$, then
\be\label{eq:loc_eq_super_bis}
\lim_{N\to+\infty}\left[
\Exp 
\psi\left(
\tau_{x_N}\eta^{\alpha,N}_{Nt}
\right)-
\int_{\overline{\mathbf X}}\psi(\eta)d\mu_c^{
\tau_{x_N}\alpha
}(\eta)
\right]=0
\ee
\end{theorem}
\begin{remark}\label{remark_strong_loc_eq_super}
 (i) For a typical sequence $(x_N)_{N\in\N}$,  \eqref{eq:loc_eq_super} 
 is true regardless of the values \eqref{limsupinf}. But if $\rho^*(t,u)<\rho_c$,  
 \eqref{eq:loc_eq_upper} is more precise and does not require the typicality assumption.\\ 
(ii)  It is possible to give an alternative formulation of 
\eqref{eq:loc_eq_super_bis}  that looks more like a ``usual'' 
local equilibrium statement in the sense that it is an actual convergence result. 
To this end, we consider subsequences of $(x_N)_{N\in\N}$ such that the shifted 
environment  $\tau_{x_N}\alpha$ 
has a limit $\overline{\alpha}$ (which, by assumption, satisfies 
conditions (i)-(ii) of Definition \ref{def_typical}). Then, equivalent to \eqref{eq:loc_eq_super_bis},
is the statement that along every such subsequence, 
\be\label{eq:loc_eq_super_ter}
\lim_{N\to+\infty}
\Exp 
\psi\left(
\tau_{x_N}\eta^{\alpha,N}_{Nt}
\right)=
\int_{\mathbf X}\psi(\eta)d\mu_c^{
\overline{\alpha}
}(\eta)
\ee
\end{remark}
The most important requirement in Definition \ref{def_typical} is (ii). 
Suppose for instance that the environment $\alpha$ is i.i.d., 
and its marginal distribution takes value $1$ with positive probability. 
Then it is a.s. possible to find sequences $(x_N)_{N\in\N}$, $(y_N)_{N\in\N}$ 
and $(z_N)_{N\in\N}$  (depending on the realization of the environment) such that 
\begin{eqnarray*}
&&\lim_{N\to+\infty}N^{-1}x_N=u, \quad y_N<x_N<z_N,\\
&&\alpha(x)=1\quad\mbox{for every  }x\in[y_N,z_N],\mbox{ and}\\
&&\!\!\!\!\!\! \lim_{N\to+\infty}(x_N-y_N)=\lim_{N\to+\infty}(z_N-x_N)=+\infty
\end{eqnarray*}
 In this case, assumption (ii) is violated because $\overline{\alpha}$ 
is the homogeneous environment given by $\overline{\alpha}(x)=1$ for all $x\in\Z$, 
hence $\liminf_{z\to-\infty}\overline{\alpha}(z)=1$.
Thus, by condition \eqref{cond_simple_bar}, the invariant measure 
$\mu^{\overline{\alpha}}_\beta$ in \eqref{eq:loc_eq_super_bis} exists not only 
for $\beta\in[0,c]$ as in the typical (in the sense of Definition \ref{def_typical}) 
situation, but for every $\beta\in[0,1)$. These measures cover the whole range 
of mean densities $\rho\in[0,+\infty)$ and we might imagine that the distribution 
around $x_N$ approaches the measure $\mu^{\overline{\alpha}}_\beta$ corresponding 
to the hydrodynamic density $\rho(t,u)$.\\ \\
To discard such behavior, we would need to control how fast site $x_N$ would ``see'' 
slow sites on the left, which eventually impose a limiting distribution with density 
$\rho_c$ corresponding to a maximal current value $(p-q)c$. Without the above typicality 
assumption, we are not able to prove such a statement at given times, but we can obtain 
a weaker time-integrated result. 
\begin{theorem}\label{prop_strong_loc_eq_super_cesaro}
 Under assumptions and notations of Theorem \ref{th_hydro}, the following holds.
Let $t>0$, and  $(x_N)_{N\in\N}$ be an {\em arbitrary} sequence of sites such that $\lim_{N\to+\infty}N^{-1}x_N=u$, where $u\in\R$ 
 is such that $\rho_*(t,u)\geq\rho_c$.   Let $\psi:\overline{\bf X}\to\R$ be a continous local function. Then
\be\label{eq:loc_eq_super_cesaro}
\lim_{\delta\to 0}\limsup_{N\to+\infty}\left|
\frac{1}{\delta}\int_{t-\delta}^{t}\Exp 
\psi\left(
\tau_{x_N}\eta^{\alpha,N}_{Nt}
\right)dt-\int_{\overline{\mathbf X}}\psi(\eta)d\mu^{
\tau_{x_N}\alpha
}_c(\eta)
\right|=0
\ee
\end{theorem}
An illuminating particular case of the above theorems, outlined in the introduction, is when the initial datum is uniform.
\begin{example}\label{ex_ill}
Assume 
\be\label{initial_uniform}\rho_0(x)\equiv\rho\ee
for some $\rho\geq 0$. Then $\rho(t,x)\equiv\rho$ for all $t>0$.
Specializing Theorems \ref{th_strong_loc_eq} and \ref{th_strong_loc_eq_super} to $u=0$, we obtain that $\eta_{Nt}^{\alpha,N}$ converges in distribution to $\mu^{\alpha,\rho}$ if $\rho<\rho_c$, or to $\mu_c^\alpha$ if $\rho\geq\rho_c$. We may in particular achieve \eqref{initial_uniform} as follows by a sequence of initial configurations $\eta^N_0=\eta_0$ independent of $N$.\\ \\
{\em (i) Stationary initial state.} Let $\eta^N_0=\eta_0\sim\mu^{\alpha,\rho}$, with $\rho<\rho_c$. Since $\mu^{\alpha,\rho}$ is an invariant measure 
for the process with generator \eqref{generator}, for every $t>0$, we have $\eta^{\alpha,N}_{Nt}\sim\mu^{\alpha,\rho}$.  As a result, the expression between brackets in \eqref{eq:loc_eq} vanishes for every $N\in\N\setminus\{0\}$ and $t\geq 0$. In this situation, we speak of {\em conservation} of local equilibrium (since \eqref{eq:loc_eq} already holds for $t=0$), but there is in fact nothing to prove, since this conservation follows form stationarity.\\ \\
{\em (ii) Deterministic initial state.} Let $\rho\geq 0$, and $\eta^N_0=\eta_0$, where $\eta_0\in{\bf X}$ is such that
\be\label{condition_eta_strong}
\lim_{n\to+\infty}\frac{1}{n}\sum_{x=-n}^0\eta_0(x)=
\lim_{n\to+\infty}\frac{1}{n}\sum_{x=0}^n\eta_0(x)=\rho
\ee
Condition \eqref{condition_eta_strong} implies that the sequence $(\eta^N_0)_{N\in\N\setminus\{0\}}$ has density profile $\rho_0(x)\equiv\rho$.
Note that this condition is similar to the one introduced in \cite{bam}  to prove convergence of the asymmetric nonzero-mean exclusion process to the product Bernoulli measure with mean density $\rho$.
\end{example}
Pushing the analysis further, 
  we can derive from Theorems \ref{th_strong_loc_eq} and \ref{th_strong_loc_eq_super} the following convergence result, with a weaker condition than \eqref{condition_eta_strong}.
\begin{theorem}\label{cor_dream_bfl}
Let $\eta_0\in{\mathbf X}$ be such that,  for some $\rho>0$, 
\be\label{condition_eta}
\lim_{n\to+\infty}\frac{1}{n}\sum_{x=-n}^0\eta_0(x)=\rho
\ee
Then $\eta_t^\alpha$ converges in distribution as $t\to+\infty$ to $\mu^{\alpha,\rho\wedge\rho_c}$.
\end{theorem}
\begin{remark} (i) In the case $\rho<\rho_c$, 
 Theorem  \ref{cor_dream_bfl} solves the  convergence  
problem posed in  \cite[pp 195-196]{bfl}.\\  
(ii) In the case $\rho\geq\rho_c$, 
 Theorem  \ref{cor_dream_bfl} is a partial improvement over \cite[Theorem 2.2]{bmrs2}.
It improves the latter in the sense that we do not need a weak convexity assumption on $g$ as in 
\cite{bmrs2}, but is it less general with respect to initial conditions, since the result 
of \cite{bmrs2} only assumed the weaker condition
\be\label{weak_condition_eta}
\liminf_{n\to+\infty}\frac{1}{n}\sum_{x=-n}^0\eta_0(x)\geq\rho_c
\ee
\end{remark}
\section{Preliminary material}\label{sec:preliminary}
We first recall some definitions and preliminary results on the graphical construction and currents from
\cite{bmrs1,bmrs2}.
\subsection{Harris construction and coupling}\label{subsec_harris}\label{subsec:Harris}
For the Harris construction of the process with infinitesimal generator \eqref{generator},
we introduce a probability space $(\Omega,\mathcal F,\Prob)$, whose
generic element $\omega$ - called a Harris system (\cite{har}) - 
is a locally finite point measure of the form
\be\label{def_omega}
\omega(dt,dx,du,dz)=\sum_{n\in\N}\delta_{(T_n,X_n,U_n,Z_n)}
\ee
 on  $(0,+\infty)\times\Z\times(0,1)\times\{-1,1\}$,   
 where $\delta_{(.)}$ denotes Dirac measure, and $(T_n,X_n,U_n,Z_n)_{n\in\N}$ is a 
 $(0,+\infty)\times\Z\times(0,1)\times\{-1,1\}$-valued sequence.  
Under the probability measure $\Prob$, $\omega$ is a Poisson measure with 
intensity 
\be\label{def_intensity}\mu(dt,dx,du,dz):=dtdx\indicator{[0,1]}(u)du\, p(z)dz\ee
 In the sequel, the notation $(t,x,u,z)\in\omega$ will mean $\omega(\{(t,x,u,z)\})=1$. 
We shall also say that $(t,x,u,z)$ is a potential jump event.\\ 
On  $(\Omega,\mathcal F,\Prob)$,  a c\`adl\`ag process $(\eta_t^\alpha)_{t\geq 0}$ 
with generator \eqref{generator} and initial configuration $\eta_0$ can be constructed
in a unique way  (see \cite[Appendix B]{bmrs2})  so that 
\be\label{rule_1}
\forall(s,x,v,z)\in\omega,\quad
v\leq \alpha(x)g\left[\eta^\alpha_{s-}(x)\right]  \Rightarrow \eta^\alpha_s=(\eta^\alpha_{s-})^{x,x+z}
\ee
and, for all $x\in\Z$ and $0\leq s\leq s'$,
\be\label{rule_2}
\omega\left(
(s,s']\times E_x
\right)=0\Rightarrow\forall t\in(s,s'],\,\eta_t(x)=\eta_s(x)
\ee
where 
\be\label{def_E_x}
E_x:=\{(y,u,z)\in\Z\times(0,1)\times\{-1,1\}:\,
x\in\{y,y+z\}
\}
\ee
(note that the inequality in \eqref{rule_1} implies $\eta_{t-}^\alpha(x)>0$, 
cf. \eqref{properties_g}, thus $(\eta^\alpha_{t-})^{x,x+z}$ is well-defined). 
Equation \eqref{rule_1} says when a potential update time gives rise to an actual jump,
 while \eqref{rule_2} states that no jump ever occurs outside potential jump events.
We can construct a mapping
\be\label{mapping}
(\alpha,\eta_0,t)\mapsto\eta_t(\alpha,\eta_0,\omega)
\ee
The mapping \eqref{mapping} allows us to couple 
an arbitrary number of processes with generator \eqref{generator}, 
corresponding to different values of $\eta_0$, by using the same 
Poisson measure $\omega$ for each of them.  Since $g$ is nondecreasing, 
the update rule \eqref{rule_1} implies that  
the mapping \eqref{mapping} is nondecreasing with respect to $\eta_0$. 
It follows that the process is \textit{completely monotone}, 
and thus attractive (see \cite[Subsection 3.1]{bgrs5}). 
For instance, the coupling of two processes $(\eta_t^{\alpha})_{t\geq 0}$ 
and $(\xi_t^{\alpha})_{t\geq 0}$
behaves as follows. Assume $(t,x,u,z)$ is a potential 
update time for the measure $\omega$ in \eqref{def_omega}, 
and assume (without loss of generality) that 
$\eta^\alpha_{t-}(x)\leq\xi^\alpha_{t-}(x)$, so that (since $g$ is nondecreasing) 
$g\left(\eta^\alpha_{t-}(x)\right)\leq g\left(\xi^\alpha_{t-}(x)\right)$. 
Then the following jumps occur at time $t$:\\ \\
(J1) If $u\leq\alpha(x)g\left(\eta^\alpha_{t-}(x)\right)$, an $\eta$ 
and a $\xi$ particle simultaneously jump from $x$ to $x+z$.\\ \\
(J2) If $\alpha(x)g\left(\eta^\alpha_{t-}(x)\right)<u\leq\alpha(x)g\left(\xi^\alpha_{t-}(x)\right)$, 
a $\xi$ particle  alone jumps from $x$ to $x+z$.\\ \\
(J3) If $\alpha(x)g\left(\xi^\alpha_{t-}(x)\right)<u$, nothing happens.\\ \\
The above dynamics implies that $\left(\eta^\alpha_t,\xi^\alpha_t\right)_{t\geq 0}$ 
is a Markov process on $\overline{\mathbf X}^2$ with generator
\begin{eqnarray}
\nonumber\widetilde{L}^\alpha f(\eta,\xi) & = & 
\sum_{x,y\in\Z}\alpha(x)p(y-x)\left(g(\eta(x))\wedge g(\xi(x))\right)\left[
f\left(\eta^{x,y},\xi^{x,y}\right)-f(\eta,\xi)
\right]\\
& + & \sum_{x,y\in\Z}\alpha(x)p(y-x)[g(\eta(x))-g(\xi(x))]^+\left[
f\left(\eta^{x,y},\xi\right)-f(\eta,\xi)
\right]\nonumber\\
& + & \sum_{x,y\in\Z}\alpha(x)p(y-x)[g(\xi(x))-g(\eta(x))]^+\left[
f\left(\eta,\xi^{x,y}\right)-f(\eta,\xi)
\right]\nonumber\\&&\label{coupled_generator}
\end{eqnarray}
 In particular, this monotonicity allows us to couple different stationary processes.
In the sequel, $(\xi^{\alpha,\rho}_s)_{s \geq 0}$ 
will denote a stationary process with initial random configuration
$\xi ^{\alpha,\rho}_0\sim\mu^{\alpha,\rho}$. 
Thanks to \eqref{stoch_inc}, by Strassen's theorem,
for any $\rho,\rho'\in[0,\rho_c)$, we may (and will in the sequel) couple $\xi_0^{\alpha,\rho}$ and $\xi_0^{\alpha,\rho'}$ in such a way that $\xi_0^{\alpha,\rho}\leq\xi_0^{\alpha,\rho'}$ if $\rho\leq\rho'$. By coupling the processes via the Harris construction, since the mapping \eqref{mapping} is nondecreasing, 
it holds almost surely that
\be\label{ordered_stat}
\rho\leq\rho'\Rightarrow\xi^{\alpha,\rho}_t\leq\xi^{\alpha,\rho'}_t,\quad\forall t\geq 0
\ee
 In the next subsection, we introduce a refined version of 
 (J1)--(J3)  above to take into account different classes of $\eta$ particles.
\subsection{Classes and coalescence}\label{sec:class}
 For two particle configurations $\eta,\xi\in{\bf\overline{X}}$, 
we say that there are discrepancies 
at site $y\in\Z$ if $\xi(y)\not=\eta(y)$. If $\xi(y)<\eta(y)$, we call them 
$(\eta-\xi)$ discrepancies, and there are
$[\eta(y)-\xi(y)]^+$ such discrepancies; 
 while if $\xi(y)>\eta(y)$, 
we call them $(\xi-\eta)$ discrepancies, and there are
$[\xi(y)-\eta(y)]^+$ such discrepancies.
We may view these collections of discrepancies as particle configurations $\beta$, $\gamma$ given 
respectively for every $z\in\Z$ by 
\be\label{def_betagamma}\beta(z)=[\eta(z)-\xi(z)]^+,\quad \gamma(z)=[\xi(z)-\eta(z)]^+
\ee
 In the above coupling  \eqref{coupled_generator}, 
the monotonicity of $g$ implies that discrepancies can never be created, while
 $\beta$ and $\gamma$ particles  annihilate if one of them jumps over the other
 (as in \cite{and} for homogeneous zero-range dynamics). 
We call this a {\em coalescence}. For instance in case  (J2)  
 in the previous subsection,  if $\eta_{t-}(x+z)>\xi_{t-}(x+z)$, 
a $\gamma$ particle at $x$ jumps over a $\beta$ particle at $x+z$ and the two of them annihilate.\\ \\
To incorporate  a decomposition of $\eta$ particles into classes,
we consider two processes $(\eta_t)_{t\geq 0}$ and $(\xi_t)_{t\geq 0}$ 
generated by \eqref{generator} and coupled  through \eqref{coupled_generator}. 
Let
\be\label{total_disc}
 \mathfrak{D}_t :=\sum_{y\in\Z}[\xi_t(y)-\eta_t(y)]^+
\ee
denote the total number of $(\xi-\eta)$ discrepancies (or $\gamma$ particles) at time $t$. 
 Since discrepancies cannot be created, 
 if we assume  $\mathfrak{D}_0$  finite, then   $\mathfrak{D}_t$ 
 is finite for all $t>0$,
and moreover  $\mathfrak{D}_t$  is nonincreasing. 
In this section
we 
interpret the decrease of  $\mathfrak{D}_t$  in terms of 
coalescences of $\beta$ and $\gamma$ particles,
and decompose $\beta$ particles into different classes that will be useful in the sequel.\\ \\
 The $\eta$-process 
is divided into $n$ classes of particles  (we do not need classes for the $\xi$-process). 
This is achieved by coupling the previous two processes to new processes $\eta^i_.$ as follows, 
for $i\in\{1,\ldots,n\}$, with $\eta^n_.=\eta_.$. 
At initial time (say $0$) we have $\eta^i_0\leq\eta^{i+1}_0$ for every $i\in\{1,\ldots,n-1\}$.
It is convenient to
view the empty configuration as a special zero-range process $\eta^0_.$ 
and to add another special zero-range process
$\eta^{n+1}_.$ with $\eta_t^{n+1}(y)=+\infty$ for all $y\in\Z$.
The configuration of $\eta$ particles of class $i\in\{1,\ldots,n\}$  is 
$\eta^{i}-\eta^{i-1}$, and $\eta^i_.$ is the configuration of $\eta$ particles of classes $1$ to $i$.
An $\eta$ particle keeps the same class during evolution. \\ \\
 Note that the occupation number of a class can be equal to zero. For instance 
if $\eta_0^i = \eta_0^{i+1}< \eta_0^{i+2}$ then there are no $\eta_0$ particles of 
class $i+1$ but there are 
$\eta_0$ particles of class $i+2$. Even if all classes of $\eta_0$ particles have 
positive occupation number it is possible that $\eta_t^i(x)=0$ for some $t>0$ and 
some $1 \leq i \leq n$. For example if 
$\eta_{t-}(x) =1$ and a  particle jumps from $x$ at time $t$ then $\eta_t^1(x) =0$.\\
Note also that the notion of priority between classes which is natural when $g \equiv 1$
(see e.g. \cite{ak}) 
is not valid for the classes defined above for non-constant $g$. \\ \\ 
Let $(t,y,u,z)$  be an atom of the measure \eqref{def_omega}. 
Let $k\in\{1,\ldots,n+1\}$  be  
such that
\be\label{classindex}
\eta^{k-1}_{t-}(y)<\xi_{t-}(y)\leq\eta^{k}_{t-}(y)
\ee
Thus for $i=1,\ldots,k-1$, there are $\eta^{i}_{t-}(y)-\eta^{i-1}_{t-}(y)$ 
particles of class $i$ at site $y$, and they are
all matched to a $\xi$ particle.
If $i=k\leq n$, there are $\eta^{k}_{t-}(y)-\eta^{k-1}_{t-}(y)$ particles 
of class $i$ in $\eta_{t-}$ at $y$, of which
$\xi_{t-}(y)-\eta^{k-1}_{t-}(y)$ particles are matched to a $\xi$ particle.
If $i=k=n+1$, there are
$\xi_{t-}(y)-\eta^{n}_{t-}(y)$ unmatched $\xi$ particles 
at $y$ (recall that $\eta^{n}_{t-}=\eta_{t-}$).\\ \\
Then the following transitions may occur for $1\leq i$:\\ \\
(i) If $i<k$ and $\alpha(y)g(\eta^{i-1}_{t-}(y))<u<\alpha(y)g(\eta^{i}_{t-}(y))$, 
a pair of matched particles, an $\eta$ particle of class $i$ and a $\xi$ particle, 
move together to $y+z$.\\ \\
(ii) If $i=k\leq n$ and $\alpha(y)g(\eta^{k-1}_{t-}(y))<u<\alpha(y)g(\xi_{t-}(y))$, 
a pair of matched particles, an $\eta$ particle of class $k$ and a $\xi$ particle, 
move together to $y+z$.\\ \\
(iii) If $i=k\leq n$ and $\alpha(y)g(\xi_{t-}(y))<u<\alpha(y)g(\eta^{k}_{t-}(y))$, 
an unmatched $\eta$ particle of class $k$ (but no $\xi$ particle) moves to $y+z$.\\ \\
(iv) If $i=k=n+1$ and $\alpha(y)g(\eta^{k-1}_{t-}(y))<u<\alpha(y)g(\xi_{t-}(y))$, 
an unmatched $\xi$ particle (but no $\eta$ particle) moves to $y+z$.\\ \\
(v)  If $k<i\leq n$ and $\alpha(y)g(\eta^{i-1}_{t-}(y))<u<\alpha(y)g(\eta^{i}_{t-}(y))$, 
an unmatched $\eta$ particle of class $i$
moves to $y+z$.\\ \\
We label $\eta$ particles increasingly from left to right within each class,  and decide  that 
whenever a particle jumps to the right (resp. left) from a given site, it is necessarily 
the particle at this (initial) site with the highest (resp.  lowest) label among all 
particles of the same class at this site.    This enables us to maintain the order of labels during the evolution. 
We initially label $\xi$ particles in an arbitrary way. Whenever a $\xi$ particle jumps from a site, if several $\xi$ particles are present at this site, the label of the jumping $\xi$ particle may be chosen arbitrarily, for instance uniformly among all $\xi$ labels.\\ \\
When $\eta_{t-}(y+z)>\xi_{t-}(y+z)$, we denote by $k'\leq n$ the
lowest class of an unmatched $\eta$ particle at $y+z$. 
Coupling between an unmatched $\eta$ particle and an unmatched $\xi$ particle 
occurs in the following cases.\\ \\
(I) If in case (iv) there are unmatched $\eta$ particles at site $y+z$,  
the $\xi$ particle that moves to $y+z$ matches with one of 
the unmatched $\eta$ particles of class $k'$.\\ \\
(II) If in cases (iii)  or (v)  there are unmatched $\xi$ particles 
at $y+z$, one of them is randomly selected 
to match with the arriving $\eta$ particle.\\ \\
Besides, we point out that:\\ \\
(III) In cases (i)-(ii) above, the $\xi$ particle that moves to $y+z$ remains 
matched to the same $\eta$ particle of class $i$ if 
$\eta_{t-}(y+z)\leq \xi_{t-}(y+z)$ or $i\leq k'$. If $i>k'$, then the arriving 
$\xi$ particle matches with an $\eta$ particle of class $k'$ (thus it leaves 
the $\eta$ particle it was previously matched with). \\ \\
The matching of $\eta$ and $\xi$ particles defined above satisfies the following property.
A $\xi$ particle cannot be matched with an $\eta$ particle of class $j$ if there 
is an unmatched  $\eta$ particle of lower class present at the site.
 This ensures that the definition \eqref{classindex} makes sense at all times $t$ and all sites $y$.\\ \\
Coming back to the quantity  $\mathfrak{D}_t$  defined in \eqref{total_disc}, we have that  $\mathfrak{D}_t-\mathfrak{D}_{t-}=-1$  if and only if 
we are in one of the above cases (I)-(II).
This corresponds to a coalescence.
\subsection{Currents}\label{subsec_currents}
Let $x_.=(x_s)_{s\geq 0}$ denote a $\Z$-valued piecewise constant c\`adl\`ag   path such 
that $\vert x_s-x_{s-}\vert\leq 1$ for all $s\geq 0$. In the sequel we will  use  
paths $(x_.)$ independent of the Harris system used for the particle dynamics,
hence we may assume that $x_.$ has no jump time in common with the latter.  
We denote by 
$\Gamma_{x_.}^\alpha(\tau,t,\eta)$ 
the rightward current across the path $x_.$ 
in the time interval $(\tau,t]$ in the quenched process  $(\eta_s^\alpha)_{s\geq \tau}$  starting from $\eta$ 
 at time $\tau$  in environment $\alpha$, that is  the sum of two contributions. The contribution of particle jumps is
 the number of times a particle jumps from $x_{s-}$ to $x_{s-}+1$ (for $\tau<s\le t$), 
 minus the number of times a particle jumps from $x_{s-}+1$ to $x_{s-}$. 
 The contribution of path motion is obtained by summing over jump times
$s$ of the path, a quantity equal to the number of particles at $x_{s-}$ if the jump is to the left, or
 minus  the number of particles at $x_{s-}+1$ if the jump is to the right.
This  can be precisely written (using notation \eqref{def_omega}),
 \begin{eqnarray}
 \Gamma^\alpha_{x_.}(\tau,t,\eta) & := &
 \int
 \indicator{
 \left\{u\leq\alpha(x_{s})g[\eta_{s-}^\alpha(x_{s})]\right\}}\indicator{\{\tau<s\leq t,\,z=1,\,x=x_{s}\}}
 \omega(ds,dx,du,dz)
 \nonumber\\ 
 &  - & \int
 \indicator{
 \left\{u\leq\alpha(x_{s}+1)g[\eta_{s-}^\alpha(x_{s}+1)]\right\}}
 \indicator{\{\tau<s\leq t,\,z=-1,\,x=x_{s}+1\}}
 \omega(ds,dx,du,dz)\nonumber\\ 
& - & \sum_{\tau<s\leq t}(x_s-x_{s-})\eta_{s}^\alpha\left[
\max(x_s,x_{s-})
\right]\label{current_harris}
 \end{eqnarray}
 If  
\be\label{finite_right}\sum_{x>x_\tau}\eta(x)<+\infty\ee
 we also have
 \be\label{current}
 \Gamma^\alpha_{x_.}(\tau,t,\eta)=\sum_{x>x_t}\eta_t^\alpha(x)-\sum_{x>x_\tau}\eta(x)
 \ee
For $x_0\in\Z$, we will write $\Gamma^\alpha_{x_0}$ for the current across the 
fixed site $x_0$; that is, $\Gamma^\alpha_{x_0}(\tau,t,\eta):=\Gamma^{\alpha}_{x_.}(\tau,t,\eta)$,
where $x_.$ is the constant path defined by $x_t=x_0$ for all  $t\geq \tau$. 
If $\tau=0$, we simply write $\Gamma_{x_.}^\alpha(t,\eta)$ or $\Gamma_{x_0}^\alpha(t,\eta)$ instead of
$\Gamma_{x_.}^\alpha(0,t,\eta)$ or  $\Gamma_{x_0}^\alpha(0,t,\eta)$.\\ \\
It follows from \eqref{current}  that if $a,b\in\Z$ and $a<b$, then
\be\label{difference_currents}
\Gamma_b^\alpha(\tau,t,\eta)-\Gamma_a^\alpha(\tau,t,\eta)=-\sum_{x=a+1}^b\eta_t(x)+\sum_{x=a+1}^b\eta(x)
\ee
 The latter formula remains valid even if \eqref{finite_right} does not hold.
 The following results    will be important tools to compare currents.
For a particle configuration  $\zeta\in\overline{\mathbf X}$  and a site $x_0\in \Z$, we define
\be\label{def_cfd}
F_{x_0}(x,\zeta):=\left\{
\ba{lll}
\sum_{y=1+x_0}^{x}\zeta(y) & \mbox{if} & x>x_0\\ \\
-\sum_{y=x}^{x_0}\zeta(y) & \mbox{if} & x\leq x_0
\ea
\right.
\ee
Let us couple two processes $(\zeta_t)_{t\geq 0}$ and $(\zeta'_t)_{t\geq 0}$ 
 through \eqref{coupled_generator},   
with $x_.=(x_s)_{s\geq 0}$ as above.  
\begin{lemma}
\label{lemma_current}
\be\label{current_comparison}\Gamma^\alpha_{x_.}(t,\zeta_0)
-\Gamma^\alpha_{x_.}(t,\zeta'_0)
\geq -\left(0\vee\sup_{x\in\Z}\left[F_{x_0}(x,\zeta_0)-F_{x_0}(x,\zeta'_0)\right]\right)\ee
\end{lemma}
\begin{corollary}\label{corollary_consequence}
For $y\in\Z$, define the configuration 
\be\label{conf-max}
\eta^{*,y}:=(+\infty)\indicator{(-\infty,y]\cap\Z}
\ee
Then,  for any $\zeta\in\overline{\mathbf X}$, 
\be\label{eq_corollary_consequence}
\Gamma^\alpha_y(t,\zeta)\leq\Gamma^\alpha_y(t,\eta^{*,y})
\ee
\end{corollary}
 Lemma \ref{lemma_current} and 
Corollary \ref{corollary_consequence} are  proved in  \cite{bmrs1}.
 The following version of {\em finite propagation property} will be 
used repeatedly in the sequel. See \cite{bmrs1}  for a proof.
\begin{lemma}\label{lemma_finite_prop}
For each  $V > 1$, there exists $b =  b(V) >  0$  such that for large enough $t$,
 if $ \eta_0 $ and $\xi _ 0 $ agree on an interval $(x,y)$, then, outside probability $e^{-bt}$, 
$$
\eta_s(u) = \xi_s(u) \quad\mbox{for all }0 \leq s \leq t \mbox{ and } u\in (x+Vt,y-Vt)
$$ 
\end{lemma}
Next corollary  to Lemma \ref{lemma_current} follows from the latter  
combined with Lemma \ref{lemma_finite_prop}.  Its proof is an
adaptation of the one of \cite[Corollary 4.2]{bmrs2}.
\begin{corollary}
\label{corollary_current}
Define
$x_t^M= \sup_{s\in[t_0,t]} x_s,
x_t^m= \inf_{s\in[t_0,t]} x_s$.
Let $\eta_0,\xi_0\in{\mathbf X}$.
Then,  given  $V>1$, 
\begin{eqnarray}
&&\Gamma^\alpha_{x_.}(t,\eta_0)-\Gamma^\alpha_{x_.}(t,\xi_0)\label{current_comparison_2}\\
&&\quad \geq -\left(0\vee\sup_{x\in [\min(x_0,x_t^m)-Vt, \max(x_0,x_t^M)+1+Vt]}
\left[F_{x_0}(x,\eta_0)-F_{x_0}(x,\xi_0)\right]\right)\nonumber
\end{eqnarray}
with 
probability tending to $1$ as $t\to+\infty$.
\end{corollary} 
The following result (see \cite[Proposition 4.1]{bmrs2}) is concerned  
with the asymptotic current produced by a source-like initial condition.
\begin{proposition}\label{current_source} 
Assume $x_t$ is such that $\lim_{t\to+\infty}t^{-1}x_t$ 
exists.  Let $\eta^{\alpha,t}_0:=\eta^{*,x_t}$, see \eqref{conf-max}.
Then 
\begin{eqnarray}
\limsup_{t\to\infty}\left\{\Exp\left|
t^{-1}\sum_{x>x_t}\eta^{\alpha,t}_t(x)-(p-q)c
\right| - p[\alpha(x_t)-c]\right\}
& \leq & 0 \label{upperbound_current_source}
\end{eqnarray}
\end{proposition}
Finally, the following result for the equilibrium current will be important for our purpose.
\begin{lemma}\label{current_critical}(\cite[Lemma 4.10]{bmrs2}).
Let $\alpha\in{\bf A}$, $\beta\in[0,\inf_x\alpha(x))$
 and
$\xi_{0,\beta}^{\alpha}\sim\mu^\alpha_\beta$.
Let $(x_t)_{t>0}$ be a $\Z$-valued 
family and assume the limit $\lim_{t\to+\infty}t^{-1}x_t$
exists. Then
\[
\lim_{t\to\infty}t^{-1}\Gamma^\alpha_{x_t}\left(t,\xi_{0,\beta}^{\alpha}\right)=(p-q)\beta
\quad\mbox{ in }\,\,L^1(\Prob_0\otimes\Prob)
\]
\end{lemma}
\section{Proof of Theorem \ref{th_strong_loc_eq}}\label{sec:proof_loc_eq}
 Recall that $\rho(.,.)$ denotes 
the entropy solution to \eqref{conservation_law} with Cauchy datum $\rho_0$, 
and $(\eta^{\alpha,N}_t)_{t\geq 0}$ denotes the process under study. 
Let $\Delta>0$, $V>1$.  In relation to the finite propagation property (Lemma \ref{lemma_finite_prop}), 
we set 
\be\label{def_t0}
t_0:=t-\Delta/( 4 V),
\ee
Let also
\begin{eqnarray}\label{def:rho1}
\rho_1 & := & 
\inf\left\{
\rho(s,y):\,
s\in[t_0,t],\,y\in[u-\Delta,u+\Delta]
\right\}\\\label{def:rho2}
\rho_2 & := & \sup\left\{
\rho(s,y):\,
s\in[t_0,t],\,y\in[u-\Delta,u+\Delta]
\right\}
\end{eqnarray}
 To prove Theorem \ref{th_strong_loc_eq}, using the couplings introduced in Subsection \ref{subsec:Harris}, we compare the process $(\eta^{\alpha,N}_.)$ 
  to the stationary processes $(\xi^{\alpha,\rho}_.)$ introduced 
  in Subsection \ref{subsec:Harris} (see \eqref{ordered_stat}) for suitable values of $\rho$. 
 Recalling the set $\mathcal R$ in  Lemma \ref{lemma_properties_flux}, 
Theorem \ref{th_strong_loc_eq} is mainly a consequence of the following result 
\begin{proposition}\label{prop_nomore_disc}
Assume  $\rho_0\in\mathcal R$   
is such that
\be\label{rho0_rho3}
\rho_0<\rho_1  
\ee
Then 
\begin{eqnarray}
\label{coupling_2}
\lim_{N\to+\infty}\Prob\left(
\eta^{\alpha,N}_{Nt}(y)\geq\xi_{Nt}^{\alpha,\rho_0}(y),\,\forall y\in\Z\cap [N(u-\Delta/4),N(u+\Delta/4)]
\right)=1
\end{eqnarray}
Similarly, if $\rho_3\in\mathcal R$ is such that 
\be\label{rho0_rho3_bis}\rho_2<\rho_3\ee
then
\begin{eqnarray}
\label{coupling_1}
\lim_{N\to+\infty}\Prob\left(
\eta^{\alpha,N}_{Nt}(y)\leq\xi_{Nt}^{\alpha,\rho_3}(y),\,
\forall y\in\Z\cap [N(u-\Delta/4),N(u+\Delta/4)]
\right)=1
\end{eqnarray}
\end{proposition}
Thus
statement \eqref{coupling_1}, resp. \eqref{coupling_2}, means that 
by time $Nt$, there is no more $(\eta^{\alpha,N}-\xi^{\alpha,\rho_3})$,
resp. $(\xi^{\alpha,\rho_0}-\eta^{\alpha,N})$ 
discrepancy in 
a  neighborhood of $\lfloor Nu\rfloor$  of size order $N$.  \\ \\
 Proposition \ref{prop_nomore_disc} will be established in the next subsections.
We now conclude the proof of Theorem \ref{th_strong_loc_eq} given the proposition. \\
\begin{proof}{Theorem}{th_strong_loc_eq}
 Since $\psi$ is local, there exists a finite subset $S$ of $\Z$ such that $\psi(\eta)$ depends only 
on the restriction of $\eta$ to $S$. Since $\psi$ is bounded, 
for all $\eta,\xi\in{\mathbf X}$, 
\be\label{lipsi}
|\psi(\eta)-\psi(\xi)|\leq 2||\psi||_\infty\sum_{x\in S}|\eta(x)-\xi(x)|
\ee
We first prove \eqref{eq:loc_eq_upper}. 
Because
$\mathcal R$ is dense in $[0,\rho_c]$ and $\rho^*(t,u)<\rho_c$, we can find $\Delta>0$ and  
$\rho_3\in\mathcal R$ arbitrarily close to $\rho^*(t,u)$ such that $\rho^*(t,u)<\rho_2<\rho_3$.
Let $E_N$ be the 
event in \eqref{coupling_1}.
 Then for $N$ large enough, 
$x_N+S\subset\Z\cap[N(u-\Delta/4),N(u+\Delta/4)]$.  
We now write
\begin{eqnarray}
\Exp\psi\left(\tau_{x_N}\eta^{\alpha,N}_{Nt}\right) 
& \leq  &
\Exp\left[1_{E_N}\psi\left(\tau_{x_N}\xi^{\alpha,\rho_3}_{Nt}\right)\right]
+||\psi||_\infty(1-\Prob(E_N))\nonumber\\
& \leq  &
\Exp\psi\left(\tau_{x_N}\xi^{\alpha,\rho_3}_{Nt}\right)+||\psi||_\infty(1-\Prob(E_N))\nonumber\\
& \leq & \int_{\mathbf X}\psi(\eta)d\mu^{\tau_{x_N}\alpha,\rho^*(t,u)}(\eta)
+2||\psi||_{\infty}|S|C(\rho_3-\rho^*(t,u))\nonumber\\
&&+||\psi||_\infty(1-\Prob(E_N))\label{wenowwrite}
\end{eqnarray}
 for some constant $C>0$. On the first line of \eqref{wenowwrite}, 
we used the assumption that $\psi$ is nondecreasing. 
The desired result follows,  since $\Prob(E_N)\to 1$ by \eqref{coupling_1}, and $\rho_3$ is arbitrarily close to $\rho^*(t,u)$.
The last inequality in \eqref{wenowwrite} follows from
\begin{eqnarray}
&&\!\!\!\!\!\!\!\!\!\!\!\!\!\!\!\!\!\!
\Exp\psi\left(\tau_{x_N}\xi^{\alpha,\rho_3}_{Nt}\right)
-\int_{\mathbf X}\psi(\eta)d\mu^{\tau_{x_N}\alpha,\rho^*(t,u)}(\eta)\nonumber\\
& = & \Exp\left[
\psi\left(\tau_{x_N}\xi^{\alpha,\rho_3}_{0}\right)-
\psi\left(\tau_{x_N}\xi^{\alpha,\rho^*(t,u)}_{0}\right)
\right]\nonumber\\
& \leq & 2||\psi||_\infty\Exp
\sum_{z\in S}\left[\xi^{\alpha,\rho_3}_{0}( x_N+ z)-\xi^{\alpha,\rho^*(t,u)}_{0}( x_N+ z)
\right]\nonumber\\
& \leq & 2||\psi||_\infty|S|C(\rho_3-\rho^*(t,u))\label{inequality_follows}
\end{eqnarray}
 where we used stationarity for the equality, and for the first inequality, we used \eqref{lipsi}, $\rho^*(t,u)<\rho_3$ and \eqref{ordered_stat}.
For the last inequality, we write the right-hand side of \eqref{inequality_follows} as 
\begin{eqnarray} 2||\psi||_\infty
\sum_{z\in S}\left\{
R\left(
\frac{\overline{R}^{-1}(\rho_3)}{\alpha(x_N+z)}
\right)-R\left(
\frac{\overline{R}^{-1}(\rho^*(t,u))}{\alpha(x_N+z)}
\right)
\right\}\label{forthelast}
\end{eqnarray}
and we use Lemma \ref{lemma_properties_flux}, $R\in C^1([0,1))$ (cf. \eqref{mean_density}),
and the fact that $\rho^*(t,u)<\rho_3<\rho_c$.  \\ \\
Now we prove
\eqref{eq:loc_eq_super_lower}. 
We can find $\Delta$ small enough and $\rho_0\in\mathcal R$ 
 arbitrarily close to $\rho_*(t,u)$  so that  $\rho_0<\rho_1<\rho_*(t,u)$. 
Let $F_N$ denote the event in \eqref{coupling_2}.  Arguing as in \eqref{wenowwrite} 
we obtain, for some other constant $C'>0$,
\begin{eqnarray}
\Exp\psi\left(\tau_{x_N}\eta^{\alpha,N}_{Nt}\right) & \geq  &
\Exp\left[1_{F_N}\psi\left(\tau_{x_N}\xi^{\alpha,\rho_0}_{Nt}\right)\right]\nonumber\\
& \geq  &
\Exp\psi\left(\tau_{x_N}\xi^{\alpha,\rho_0}_{Nt}\right)-||\psi||_\infty(1-\Prob(F_N))
\nonumber\\
& \geq & \int_{\mathbf X}\psi(\eta)d\mu^{\tau_{x_N}\alpha,\rho_*(t,u)}(\eta)
-2||\psi||_{\infty}|S|C'(\rho_*(t,u)-\rho_0)\nonumber\\
&&-||\psi||_\infty(1-\Prob(F_N))
\end{eqnarray}
which implies \eqref{eq:loc_eq_super_lower}, since $\Prob(F_N)\to 1$ by \eqref{coupling_2}.\\ \\
We finally prove \eqref{eq:loc_eq}.
Since $\rho(t,u)<\rho_c$, $\rho(.,.)$ is continuous 
at $(t,u)$, and $\mathcal R$ is dense, we can find $\Delta>0$ and $\rho_0, \rho_3\in\mathcal R$ such that
$\rho_0<\rho_1<\rho_2<\rho_3$ and $\rho_3-\rho_0$ is arbitrarily small.
We now write
\begin{eqnarray}
&&\!\!\!\!\!\!\!\!\!\!\!\!\!\!\!\!\!\!\left|
\Exp\psi\left(\tau_{x_N}\eta^{\alpha,N}_{Nt}\right)
-\int_{\mathbf X}\psi(\eta)d\mu^{\tau_{x_N}\alpha,\rho(t,u)}(\eta)
\right|\nonumber\\ &\leq  & \left|
\Exp\psi\left(\tau_{x_N}\eta^{\alpha,N}_{Nt}\right)
-\Exp\psi\left(\tau_{x_N}\xi^{\alpha,\rho_0}_{Nt}\right)
\right|\nonumber\\
&&+ \left|
\Exp\psi\left(\tau_{x_N}\xi^{\alpha,\rho_0}_{Nt}\right)
-\int_{\mathbf X}\psi(\eta)d\mu^{\tau_{x_N}\alpha,\rho(t,u)}(\eta)
\right|  \label{using_lipsi}
\end{eqnarray}
Proceeding as in \eqref{inequality_follows}, the second term on the r.h.s. 
of \eqref{using_lipsi} is bounded by 
\[
2||\psi||_\infty|S|C''(\rho(t,u)-\rho_0)\leq 2||\psi||_\infty|S|C''(\rho_3-\rho_0)
\]
For the first term on the r.h.s. of \eqref{using_lipsi}, we insert the indicator  
of $E_N\cap F_N$. Since on this event we have 
$\xi^{\alpha,\rho_0}_{Nt}\leq\eta^{\alpha,N}_{Nt}\leq \xi^{\alpha,\rho_3}_{Nt}$, 
using \eqref{lipsi} again, we can bound this term by 
\[
2||\psi||_\infty|S|C'''(\rho_3-\rho_0)+2||\psi||_\infty\left[1-\Prob(E_N\cap F_N)\right]
\]
This concludes the proof of \eqref{eq:loc_eq}.
  \end{proof}
\subsection{Plan of proof of Proposition \ref{prop_nomore_disc}}\label{subsubsec:planproof}
 {}From now on in this section, to lighten the notation, since the disorder $\alpha$
is fixed, we write $(\eta^{N}_s)_{s \geq 0}$ instead of
$(\eta^{\alpha,N}_s)_{s \geq 0}$, and
$(\xi^{\rho}_s)_{s \geq 0}$  (for $\rho\geq 0$) instead of
$(\xi^{\alpha,\rho}_s)_{s \geq 0}$. \\ 
We will derive 
Proposition \ref{prop_nomore_disc} from a similar result in which $\eta^N_.$ 
and  $\xi^{\rho}_.$  are replaced by processes 
$\widetilde{\eta}^N_.$ and  $\widetilde{\xi}^{\rho}_.$  defined as follows.
At time $t_0$, we define  truncated versions of 
$\eta^N_{Nt_0}$ and 
$\xi^{\rho}_{Nt_0}$ around site $\lfloor Nu\rfloor$ by 
\begin{eqnarray}
\widetilde{\eta}^N_{Nt_0}(y) & := & \eta^N_{Nt_0}(y)\indicator{
\left\{
\Z\cap[N(u-\Delta/2),N(u+\Delta/2)]
\right\}
}(y)\label{def_etatilde}\\
\widetilde{\xi}^{\rho}_{Nt_0}(y) & := & \xi^\rho_{Nt_0}(y)\indicator{
\left\{
\Z\cap[N(u-\Delta/2),N(u+\Delta/2)]
\right\}
}(y)\label{def_xitilde}
\end{eqnarray}
Then, on the time interval $[Nt_0,Nt]$, we define $\widetilde{\eta}^N_.$ 
and  $\widetilde{\xi}^\rho_.$  as the evolved processes starting from the above 
initial configurations. {}From now on, our purpose will be to establish the following result.
\begin{proposition}\label{prop_nomore_disc_2}
Let $\Delta>0$, $\rho_0$, $\rho_1$, $\rho_2$, $\rho_3$ 
be as in Proposition \ref{prop_nomore_disc}. 
Then
\begin{eqnarray}
\label{coupling_1_tilde}
\lim_{N\to+\infty}\Prob\left(
\widetilde{\eta}^N_{Nt}(y)\leq{\xi}^{\rho_3}_{Nt}(y),\quad\forall y\in\Z
\right)=1\\
\label{coupling_2_tilde}
\lim_{N\to+\infty}\Prob\left(
{\eta}^N_{Nt}(y)\geq\widetilde{\xi}_{Nt}^{\rho_0}(y),\quad\forall y\in\Z
\right)=1
\end{eqnarray}
\end{proposition}
Recalling the definition \eqref{def_t0} of $t_0$, Proposition \ref{prop_nomore_disc} 
is deduced from Proposition \ref{prop_nomore_disc_2} by applying the
finite propagation property (Lemma \ref{lemma_finite_prop}) to the 
pairs 
$(\eta^N_.,\widetilde{\eta}^N_.)$
and  
$(\xi^\rho_.,\widetilde{\xi}^\rho_.)$
 for $\rho\in\{\rho_0,\rho_3\}$.  \\ \\
To prove Proposition \ref{prop_nomore_disc_2}, the main step will be to establish 
a weaker statement, namely that the number of $({\eta}-\widetilde{\xi}^{\rho_3})$, resp. $(\widetilde\xi^{\rho_0}-\widetilde\eta)$ discrepancies, is $o(N)$:
\begin{proposition}\label{prop_mean_disc}
Let $\Delta>0$, $\rho_0$, $\rho_1$, $\rho_2$, $\rho_3$ be as in Proposition \ref{prop_nomore_disc}.
Then, for every $s\in(t_0,t]$ and $\iota>0$,
\begin{eqnarray}
\label{coupling_1_mean}
\lim_{N\to+\infty}\Prob\left(\left\{
(2\Delta N)^{-1}\sum_{
y\in\Z
}
\left[
\widetilde{\eta}^N_{Ns}(y)-{\xi}^{\rho_3}_{Ns}(y)
\right]^+>\iota
\right\}\right)=0\\
\label{coupling_2_mean}
\lim_{N\to+\infty}\Prob\left(\left\{
(2\Delta N)^{-1}\sum_{
y\in\Z
}
\left[
\widetilde{\xi}^{\rho_0}_{Ns}(y)-{\eta}^N_{Ns}(y)
\right]^+>\iota
\right\}\right)=0
\end{eqnarray}
\end{proposition}
The final step, performed in Subsection \ref{subsec:prop_implies}, 
will be to show that Proposition \ref{prop_mean_disc} implies Proposition 
\ref{prop_nomore_disc_2}.  To this end, we will introduce an intermediate time
\be\label{def_t1} t_1:=t-\frac{\Delta}{ 8 V}\ee
and show that the presence of only $o(N)$ discrepancies at time $Nt_1$, given by Proposition \ref{prop_mean_disc} for $s=t_1$, actually implies that no more discrepancy remains at time $Nt$. \\ \\
{}From now on, we will concentrate on the proof of \eqref{coupling_2_mean}, 
\eqref{coupling_2_tilde} and \eqref{coupling_2}, the proofs of 
\eqref{coupling_1_mean}, \eqref{coupling_1_tilde} and \eqref{coupling_1} being similar. 
In order to prove  \eqref{coupling_2_mean}, for $s\in[t_0,t]$, we set
(recall definition \eqref{def_betagamma} of $\beta$ and $\gamma$ particles) 
\be\label{config_disc}
\widetilde{\beta}^N_{Ns}(y):=\left[{\eta}^N_{Ns}(y)-\widetilde{\xi}^{\rho_0}_{Ns}(y)\right]^+,\,
\widetilde{\gamma}^N_{Ns}(y):=\left[\widetilde{\xi}^{\rho_0}_{Ns}(y)
-{\eta}^N_{Ns}(y)\right]^+,\,y\in\Z 
\ee
and
\begin{eqnarray}
\widetilde{e}_N(s) & := & 
(2\Delta N)^{-1}\sum_{x\in\Z}\widetilde{\gamma}^N_{Ns}(x)
\label{def_etilde}
\end{eqnarray}
Next proposition, which is the core of our argument, 
studies the time evolution of  
$\widetilde{e}_N$. 
For its statement, we introduce the following quantities.
 Choose $\varepsilon>0$ such that 
\be\label{choose_m} 2\varepsilon<\min(\rho_3-\rho_2,\rho_1-\rho_0)\ee
and $f$ is uniformly concave or convex on 
$[\rho_0,\rho_0+2\varepsilon]$  or $[\rho_3 -2 \varepsilon, \rho_3]$. 
Note that the latter requirement can be satisfied by Lemma \ref{lemma_properties_flux}, because we assumed  $\rho_3, \rho_0\in\mathcal R$. 
Let $\delta>0$ such that $2\Delta$ is a multiple of $\delta$. 
As detailed in Subsection \ref{subsec:outline} below, the microscopic 
site interval $\Z\cap[N(u-\Delta),N(u+\Delta)]$ will be divided into subintervals 
of equal length  $N\delta$ 
to obtain the following estimate.
\begin{proposition}\label{prop_evol_disc_finite}
There exist constants $A>0$ and $B>0$ independent of 
 $\varepsilon$  and  $\Delta$ such that, 
for every $s\in[t_0,t)$, with probability tending to $1$ as $N\to+\infty$,
\be\label{limit_disc}
\widetilde{e}_N(s)-\widetilde{e}_N\left(s+A\frac{\delta}{\varepsilon}\right)\geq B\varepsilon 
\widetilde{e}_N(s)
\min(\widetilde{e}_N(s),\varepsilon)
\ee
\end{proposition}
The proof of Proposition \ref{prop_evol_disc_finite} will be carried out in Subsection \ref{subsec:outline}.
The proof of Proposition \ref{prop_mean_disc} from Proposition \ref{prop_evol_disc_finite} follows.
\\
\begin{proof}{Proposition}{prop_mean_disc}
 Let $t_0<\tau<s$  and $n\in\N$. We choose $\varepsilon<\iota$, $\delta>0$ 
such that  $A\delta/\varepsilon=(s-\tau)/n$,  and divide the time interval  $[\tau,s]$ 
into $n$ intervals $[t_k,t_{k+1}]$ of length $A\delta/\varepsilon$, where $k=0,\ldots,n-1$.
Let $E_{N,n}$ denote the event that \eqref{limit_disc} holds for every 
$k=0,\ldots,n-1$, and $F_N$ the event that $\widetilde{e}_N(s)>\varepsilon$.
Since $\widetilde{e}_N$ is nonincreasing (recall from Subsection  \ref{sec:class} 
that no discrepancy is created), on $E_{N,n}$ we have $\widetilde{e}_N(t_k)>\varepsilon$ 
for every $k=0,\ldots,n-1$. Thus, on $E_{N,n}\cap F_N$,
\[
\varepsilon<\widetilde{e}_N(s)<\widetilde{e}_N(\tau)-B\varepsilon\sum_{k=0}^{n-1}\widetilde{e}_N(t_k)
\min\left(\varepsilon,\widetilde{e}_N(t_k)\right)<
\widetilde{e}_N(\tau)-nB\varepsilon^3
\]
Since
\[
\widetilde{e}_N(\tau)\leq(2\Delta N)^{-1}\sum_{x\in\Z}{\eta}^N_{N\tau}(x)+
(2\Delta N)^{-1}\sum_{x\in\Z}\widetilde{\xi}^{\rho_0}_{N\tau}(x)
\]
and each term on the above r.h.s. has a limit in law (for the first one this follows from Theorem 
\ref{th_hydro}, and for the second one from the law of large numbers), we have 
\[
\lim_{n\to+\infty}\limsup_{N\to+\infty}\Prob(E_{N,n}\cap F_N)\leq
\lim_{n\to+\infty}\limsup_{N\to+\infty}\Prob\left(
\varepsilon<\widetilde{e}_N(\tau)-nB\varepsilon^3
\right)=0
\]
On the other hand, by Proposition \ref{prop_evol_disc_finite}, 
$\lim_{N\to+\infty}\Prob(E_{N,n})=1$ for every $n\in\N$. Since
\[
\Prob(F_N)\leq \Prob(E_{N,n}\cap F_N)+\Prob(E_{N,n}^c),
\]
letting $N\to+\infty$ and then $n\to+\infty$ yields $\lim_{N\to+\infty}\Prob(F_N)=0$, 
which implies the desired result.
\end{proof}
\subsection{Proof of Proposition \ref{prop_evol_disc_finite}}\label{subsec:outline}
Recall from \eqref{config_disc} that $({\eta}-\widetilde{\xi}^{\rho_0})$ 
discrepancies are $\widetilde{\beta}$ particles, and 
$(\widetilde{\xi}^{\rho_0}-{\eta})$ discrepancies are $\widetilde{\gamma}$ particles. 
Let, for $s\in[t_0,t]$,
\be\label{def_etilden}
\widetilde{\mathfrak{e}}_N(s):=
(2\Delta N)^{-1}\sum_{z\in\Z\cap[N(u-\Delta),N(u+\Delta)]}\widetilde{\gamma}^N_{Ns}(z)
\ee
By finite propagation property (Lemma \ref{lemma_finite_prop}), we have
\be\label{mathfrak}
\lim_{N\to+\infty}\Prob\left(\left\{\widetilde{\mathfrak{e}}_N(s)=\widetilde{e}_N(s),\,
\forall s\in[t_0,t]\right\}\right)=1
\ee
We  saw  in Subsection \ref{sec:class} that whenever a $\widetilde{\beta}$ 
particle and a $\widetilde{\gamma}$ particle  coalesce, both types of 
discrepancies are killed, and the sum in \eqref{def_etilden}  decreases by one unit.
Here is the general idea of the proof.
By \eqref{rho0_rho3}, the density profile of the ${\eta}$ 
particles dominates that of the $\widetilde{\xi}^{\rho_0}$ 
particles over the space interval $[N(u-\Delta),N(u+\Delta)]$. Using this, 
everywhere along this interval, one can find a fair amount of 
$\widetilde{\beta}$ particles. On the other hand, assuming  $\widetilde{\mathfrak{e}}_N(s)$ 
not too small allows us to find 
many intervals with a good amount of $\widetilde{\gamma}$ particles. 
We will show that for suitably chosen $A$, around each such interval, 
a reasonable number of coalescences will occur in the time interval 
$[Ns,N(s+A\delta/\varepsilon)]$ between $\widetilde{\beta}$ and 
$\widetilde{\gamma}$ particles.
Each coalescence makes the sum in \eqref{def_etilden} decrease by one unit. 
This leads to the estimate \eqref{limit_disc}. Details of this scheme 
are now presented in four steps involving lemmas proved in the next subsection.\\ \\
{\em Step one: finding $\widetilde{\gamma}$ particles.} 
We divide $[N(u-\Delta),N(u+\Delta)]$ into subintervals $NI_k$ of length $N\delta$ by setting
\be\label{def:xk}
I_k=[x_k,x_{k+1}],\quad x_k=u-\Delta+k\delta,\quad k=0,\ldots,\frac{2\Delta}{\delta}-1=:K-1
\ee
The following lemma enables us to find a good number of intervals $NI_k$ 
separated by a sufficient distance (we shall see below why this is important), 
inside each of which the average number of $\widetilde{\gamma}$ particles
\be\label{def_etildenlk}
\widetilde{\mathfrak{e}}_{N,k}(s):=(\delta N)^{-1}\sum_{z\in\Z\cap N I_k}\widetilde{\gamma}^N_{Ns}(z)
\ee
 is at least half the global average in \eqref{def_etilden}.
\begin{lemma}
\label{lemma_intervals}
There exists a constant $C_1>0$ such that the following holds: 
given $n\in\N$, there exists $l^*\in\{0,\ldots,n-1\}$ such that, 
with probability tending to $1$ as $N\to+\infty$,
\be\label{intervals_discrepancies}
\left|
K_{l^*}
\right|
\geq \frac{
\Delta \widetilde{\mathfrak{e}}_N(s)}{n\delta C_1}
\ee
where, for $l\in\{0,\ldots,n-1\}$,
\be\label{def_goodset}
K_l:=\left\{
k\in\left\{ 0,\ldots,\left\lfloor\frac{K-1-l}{n}\right\rfloor\right\}:\,\widetilde{\mathfrak{e}}_{N,kn+l}(s)\geq\frac{\widetilde{\mathfrak{e}}_N(s)}{2}
\right\}
\ee
\end{lemma}  
The subintervals 
\begin{equation}\label{def:Jkl}
J_{k}^{l^*}:=I_{nk+l^*}, \quad\hbox{ for }\quad 
k=0,\ldots,\left\lfloor\frac{K-l^*-1}{n}\right\rfloor
\end{equation}
are separated from one another by a distance at least $(n-1)\delta$. 
We will eventually choose $n$  (see \eqref{no_interaction} below)  so that by finite propagation property 
(Lemma \ref{lemma_finite_prop}) subintervals do not interact over a small time interval, and
thus the overall number of coalescences in $[N(u-\Delta),N(u+\Delta)]$ 
is at least the sum of the number of coalescences around each subinterval, 
because no coalescence is ever counted twice.  To estimate this number, we proceed as follows.\\ \\
{\em Step two: dividing $\widetilde{\beta}$ particles into classes.} 
By \eqref{def:rho1}, \eqref{rho0_rho3} and \eqref{choose_m},
everywhere on the macroscopic interval $[u-\Delta,u+\Delta]$, 
the difference between the local densities of ${\eta}$ and $\widetilde{\xi}$ particles
is bounded below (with high probability) by $3\varepsilon/2$. 
Our next lemma shows that this difference  enables us   
to construct intermediate configurations ${\eta}^1$, 
${\eta}^2$, ${\eta}^3$ 
with (almost)  homogeneous profiles of densities 
$\rho^1, \rho^2, \rho^3$ such that  
$
{\eta}^1\leq{\eta}^2\leq{\eta}^3
$ 
and,  for $1\le i\le 3$,  
\begin{equation}\label{def:rho^i_1}
\rho^i=\rho_0+i\frac{\varepsilon}{2}
\end{equation}
\begin{lemma}\label{lemma_classes}
Let $h>0$,
and $\mathcal I$ be a partition of  $[u-\Delta,u+\Delta]$ 
into subintervals of length  $h$. Then, on an event of probability 
tending to $1$ as $N\to+\infty$, there exist configurations 
${\eta}^{i,N}_{Ns}$ for $i=1,\ldots,4$ with the following properties:
(i) ${\eta}^{4,N}_{Ns}={\eta}^N_{Ns}$; 
(ii) ${\eta}^{i,N}_{Ns}\leq{\eta}^{i+1,N}_{Ns}$ 
for $i\in\{1,2,3\}$; (iii) defining for $i\in\{2,3,4\}$ the configuration
of ${\eta}$ particles of class $i$ by
\[
\widetilde{\beta}^{i,N}_{Ns}:={\eta}^{i,N}_{Ns}-{\eta}^{i-1,N}_{Ns}
\]
 we have
\[
\sum_{i=2}^4\widetilde{\beta}^{i,N}_{Ns}\leq\widetilde{\beta}^N_{Ns}
\]
(iv) For every interval $I\in\mathcal I$, and for $\rho^i$ defined in \eqref{def:rho^i_1},
\be\label{density_classes}
\displaystyle \lim_{N\to+\infty}(Nh)^{-1}\sum_{x\in NI}{\eta}^{i,N}_{Ns}(x)=\rho^i
\ee
\end{lemma}
{\em Step three: class ``velocities''.} 
 The following
lemma states that, modulo explicit error bounds,
second, resp. third class ${\eta}$ particles, move at asymptotic speeds $v_2$, resp. $v_3$, defined by
\be\label{def:v_i}
v_i:=[f(\rho^i)-f(\rho^{i-1})]/(\rho^i-\rho^{i-1})
\ee 
\begin{lemma}\label{lemma_mut1}
 Let $\tau_0>0$, $\tau>0$ and $u_0\in\R$.
Set
\begin{eqnarray}
v'_2 & := & v_2-\frac{8h}{\tau\varepsilon}(\rho_0+2\varepsilon)\label{def_vprime2}\\
v'_3 & := & v_3+ \frac{8h}{\tau\varepsilon}(\rho_0+2\varepsilon)\label{def_vprime3}
\end{eqnarray}
Then, with probability tending to $1$ as $N\to+\infty$, the following events hold:
(i) all the second class ${\eta}$ particles which at time $N\tau_0$ were in $[Nu_0,+\infty)$ 
are in $[N(u_0+v'_2\tau),+\infty)$ at time $N(\tau_0+\tau)$, and (ii) all the third class ${\eta}$ particles which at time $N\tau_0$ were in $(-\infty,Nu_0]$ 
are in $(-\infty,N(u_0+v'_3\tau)]$ at time $N(\tau_0+\tau)$.
\end{lemma}
 Recall that  the flux function $f$ is $C^2$ and (see 
 (ii) of Lemma \ref{lemma_properties_flux})
uniformly concave or convex on $[\rho_0,\rho_0+2\varepsilon]$.
Without loss of generality, we will assume in the sequel that $f$ is uniformly {\em concave}
on $[\rho_0,\rho_0+2\varepsilon]$. All subsequent developments can be translated in a natural way 
to the case of uniform convexity. Thus, 
there exists a constant $C'=C'(f,\rho_0)$ 
such that
for $\varepsilon$ small enough (that is for $m$ large enough), 
\be\label{speeds}
v_3<v_2-C'\varepsilon
\ee 
 Given three adjacent intervals of equal length chosen below, we will need to show that after a suitable time, any $\widetilde{\gamma}$ particle in the central interval has been crossed either by all the second class ${\eta}$
particles in the leftmost interval, or by all the third class ${\eta}$ in the rightmost interval. This is possible because by \eqref{speeds}, we can make $v'_3$ sufficiently smaller than $v'_2$ for small $h$. Precisely:
\begin{lemma}\label{lemma_mut2}
Let $\tau_0>0$, $\tau>0$, $L>0$, $u_0\in\R$. Assume that
\be\label{cond_ctilde_h}\frac{\varepsilon\tau}{L}C'>6,\quad
h\leq \frac{C'\tau\varepsilon^2}{32(\rho_0+2\varepsilon)}
\ee
where $C'$ is the constant in \eqref{speeds}. Then $v'_2$ and $v'_3$ defined by \eqref{def_vprime2}--\eqref{def_vprime3}  satisfy
\be\label{vprimes_satisfy}
u_0+v'_2\tau>u_0+3L+v'_3\tau
\ee
\end{lemma}
Lemmas \ref{lemma_mut1} and \ref{lemma_mut2} imply that, with probability tending to $1$ as $N\to+\infty$,  
all the second class ${\eta}$ particles which at time $N\tau_0$ were in $[Nu_0,N(u_0+L)]$ 
are at time $N(\tau_0+\tau)$ to the right of all the third class ${\eta}$ particles which at time $N\tau_0$ were in $[N(u_0+2L),N(u_0+3L)]$. \\ \\
 Recall now the coalescence rules described in Subsection 
\ref{sec:class} between $\widetilde{\xi}$ particles and different 
classes of ${\eta}$ particles.  A consequence of 
Lemmas \ref{lemma_mut1} and \ref{lemma_mut2} is the following. 
\begin{corollary}\label{cor_mut1}
 Under conditions \eqref{cond_ctilde_h}, 
with probability tending to $1$ as $N\to+\infty$, the following holds 
during the time interval $[\tau_0,\tau_0+\tau]$: either (i) all the 
$\widetilde{\gamma}$ particles in $[N(u_0+L),N(u_0+2L)]$ have coalesced, 
or (ii) all the $\widetilde{\beta}^2$ particles in $[Nu_0,N(u_0+L)]$ have 
coalesced at least once, or (iii) all the $\widetilde{\beta}^3$ particles 
in $[N(u_0+2L),N(u_0+3L)]$ have coalesced at least once.
\end{corollary}
{\em Step four: counting coalescences.}
 We apply Corollary \ref{cor_mut1} as follows. First we take 
\[
L=\delta, \quad\tau_0=s,\quad\tau=C_2\delta/\varepsilon
\]
Choices of  $C_2$ and $h$ are then dictated by
condition \eqref{cond_ctilde_h}, which becomes here
\[ C_2>6/C',\quad h<C_2C'\delta\varepsilon/[32(\rho_0+2\varepsilon)]\]
Then we take $[u_0+L,u_0+2L]=J_k^{l^*}$ given by \eqref{def:Jkl},
hence $[u_0,u_0+L]=J_{k}^{l^*-1}$ and $[u_0+2L,u_0+3L]=J_{k}^{l^*+1}$.
In  Lemma \ref{lemma_intervals},  we choose $n=n^*$ defined in \eqref{no_interaction} below so that two successive intervals $J_{k}^{l^*}$ and $J_{k+1}^{l^*}$ (which are separated by a distance $(n-1)\delta$) do not interact
on a time interval of length $\tau=C_2\delta/\varepsilon$. In view of the finite propagation property (Lemma \ref{lemma_finite_prop}), the condition for this is
\be\label{no_interaction}
(n-1)\delta>2VC_2\frac{\delta}{\varepsilon}\Leftrightarrow 
n\geq n^*:=1+\left\lfloor\frac{2VC_2}{\varepsilon}\right\rfloor,
\ee
 By Corollary \ref{cor_mut1}, 
in the time interval $[s,s+C_2\delta/\varepsilon]$, 
a number of coalescences at least
\be\label{coalescence_k}
\mathfrak{n}_k^{l^*}:=\min\left(
\sum_{x\in J_{k}^{l^*-1}}\widetilde{\beta}^{N,2}_{Ns}(x),
\sum_{x\in J_{k}^{l^*}}\widetilde{\gamma}^N_{Ns}(x),
\sum_{x\in J_{k}^{l^*+1}}\widetilde{\beta}^{N,3}_{Ns}(x)
\right)
\ee
has occurred involving $\widetilde{\gamma}$ particles in the middle interval  $J^{l^*}_k$, or $\widetilde{\beta}$ particles
in the surrounding intervals  $J^{l^*\pm 1}_k$. 
 It follows from definition \eqref{no_interaction} of $n^*$ and finite propagation property  
that no $\widetilde{\beta}$ or $\widetilde{\gamma}$ 
particle can be involved simultaneously in coalescences for different values of $k$.
Hence, the total number of coalescences in the time interval $[s,s+C_2\delta/\varepsilon]$ 
is at least
the sum (over $k$) of $\mathfrak{n}_k^{l^*}$ in \eqref{coalescence_k}. Besides, 
by Lemma \ref{lemma_classes} and \eqref{def:rho^i_1}, with probability tending to 
$1$ as $N\to+\infty$, we have
\be\label{many_betas}
\min\left(
\sum_{x\in J_{k}^{l^*-1}}\widetilde{\beta}^{N,2}_{Ns}(x),
\sum_{x\in J_{k}^{l^*+1}}\widetilde{\beta}^{N,3}_{Ns}(x)
\right)\geq \frac{N\delta\varepsilon}{4}
\ee
On the event where \eqref{mathfrak}, \eqref{many_betas} and 
Lemma \ref{lemma_intervals} hold simultaneously,  for $\varepsilon$ small enough,  we have
\begin{eqnarray*}
\sum_{k=0}^{\left\lfloor\frac{K-1-l^*}{ n^*}\right\rfloor}\mathfrak{n}_k^{l^*}
&\geq& N\frac{\Delta\widetilde{\mathfrak{e}}_N(s)}{ n^*\delta C_1}\min\left(\frac{\delta\varepsilon}{4},\frac{\delta\widetilde{\mathfrak{e}}_N(s)}{2}\right)\\
&&=
N\frac{\Delta\widetilde{e}_N(s)}{\delta C_1}\frac{\varepsilon}{ 3V C_2}\min\left(\frac{\delta\varepsilon}{4},\frac{\delta\widetilde{{e}}_N(s)}{2}\right)
\end{eqnarray*}
 where we used $n^*\leq 3VC_2/\varepsilon$ for small enough $\varepsilon$.  
Since every coalescence makes the sum in \eqref{def_etilden} decrease by one unit, the result follows
with $A=C_2$ and  $B=\Delta/(24VC_1C_2)$. 
\subsection{Proposition \ref{prop_mean_disc} implies Proposition \ref{prop_nomore_disc_2}}
\label{subsec:prop_implies}
 As announced we will prove \eqref{coupling_2_tilde}, the proof of \eqref{coupling_1_tilde} being similar.
We examine the evolution of our system between times $Nt_1$ and $Nt$, where 
$t_1$ was defined in \eqref{def_t1}.
As in \eqref{def:xk}, we divide the space interval $[u-\Delta,u+\Delta]$ into $K$ 
successive subintervals $I_k$, but now for the length of these subintervals we choose 
\be\label{new_choice}
\delta:=\varepsilon\Delta/(\kappa V),
\ee
where $\kappa$ is a constant such that 
\be\label{kappa}2\kappa V/\varepsilon\in\N,\quad \kappa>\frac{48}{C'}\ee
 (the first condition means that $[u-\Delta,u+\Delta]$ can indeed be divided into subintervals of length $\delta$).
Using Lemma \ref{lemma_classes} for  $s=t_1:=\Delta/(8V)$ (cf. \eqref{def_t1}),  we divide 
${\eta}$ particles at time $Nt_1$ into four classes.
We apply Corollary \ref{cor_mut1} with 
\be\label{newchoice_lemma}
\tau_0=t_1, \quad
\tau=\Delta/( 8 V),\quad
L:=\delta=\varepsilon\Delta/(\kappa V),\quad
[u_0,u_0+L]=I_k
\ee
This is possible because, thanks to \eqref{new_choice}--\eqref{newchoice_lemma},
the first condition in \eqref{cond_ctilde_h} is satisfied. 
By Proposition \ref{prop_mean_disc}, the probability of events 
(ii) and (iii) in Corollary \ref {cor_mut1} vanishes as $N\to+\infty$, as either of 
these events implies the existence of order $N\widetilde{\gamma}$ particles at time $t_1$.
Hence event (i), that is the coalescence of all $\widetilde{\gamma}$ particles located at time $Nt_1$ in $I_k$, has probability tending to $1$ as $N\to+\infty$. Since (by \eqref{new_choice}) the number of these intervals
remains fixed as $N\to+\infty$, a union bound implies the desired conclusion.
\subsection{Proofs of lemmas}
\label{subsec:mean}
\begin{proof}{Lemma}{lemma_intervals}
 Since $\widetilde{\gamma}^N_{Ns}\leq\widetilde{\xi}^{\rho_0}_{Ns}$, 
by the law of large numbers for $\widetilde{\xi}^{\rho_0}_{Ns}\sim\mu^{\alpha,\rho_0}$, 
 there exists a constant $C_1>0$ such that, with probability tending to $1$
as $N\to+\infty$, 
\be\label{bounded_density}
(N\delta)^{-1}\sum_{z=y}^{y+\lfloor N\delta\rfloor}\widetilde{\gamma}^N_{Ns}(z)\leq C_1
\ee 
for every $y\in\Z$ such that 
$N(u-\Delta)\leq y\leq y+\lfloor N\delta\rfloor \leq N(u+\Delta)$.\\
For $l\in\{0,\ldots,n-1\}$, let
\be\label{def:Dlk-Dl}
\widetilde{\mathfrak{e}}_N^l(s) := \sum_{k=0}^{\left\lfloor\frac{K-1-l}{n}\right\rfloor}
\delta\widetilde{\mathfrak{e}}_{N, kn+l}(s)
\ee
Since  
\[
\mathfrak{e}_N(s)=(2\Delta)^{-1}\sum_{l=0}^{n-1}\widetilde{\mathfrak{e}}_N^l(s),
\]
there exists $l=l^*\in\{0,\ldots,n-1\}$ such that
\be\label{exists_such_that}
\widetilde{\mathfrak{e}}_N^{l^*}(s)\geq \frac{2\Delta}{n}\widetilde{\mathfrak{e}}_N(s)
\ee
Then 
\be\label{D-a}
\widetilde{\mathfrak{e}}_N^{l^*}(s)  =  
\sum_{k\in K_{l^*}}\delta\widetilde{\mathfrak{e}}_{N, kn+l^*}(s)
+\sum_{k\in B_{l^*}}\delta \widetilde{\mathfrak{e}}_{N, kn+l^*}(s)
\ee
where  $K_{l^*}$ was defined in  \eqref{def_goodset}, and $B_{l^*}$ is its complement in 
$\{0,\ldots,\lfloor\frac{K-1-l^*}{n}\rfloor\}$.
By \eqref{bounded_density},
\be\label{D-b}
\sum_{k\in K_{l^*}}\widetilde{\mathfrak{e}}_{N, kn+l^*}(s)\leq C_1 |K_{l^*}|
\ee
while by definition of $B_{l^*}$, we have
\be\label{D-c}
\sum_{k\in B_{l^*}}\widetilde{\mathfrak{e}}_{N, kn+l^*}(s)\leq |B_{l^*}|\frac{\widetilde{\mathfrak{e}}_{N}(s)}{2}
\ee
Thus,  by  \eqref{exists_such_that},  \eqref{D-a}, \eqref{D-b}, \eqref{D-c}, 
\[
C_1|K_{l^*}|\delta+|B_{l^*}|\delta
\frac{
\widetilde{\mathfrak{e}}_{N}(s)
}{2}\geq 
\frac{2\Delta}{n}\widetilde{\mathfrak{e}}_{N}(s)
\]
Since (recall the definition of $K$ in \eqref{def:xk}) $|B_{l^*}|\leq 2\Delta/(n\delta)$, the result follows.
\end{proof}
\mbox{}\\ \\
\begin{proof}{Lemma}{lemma_classes}
Let $E_N$ denote the event that
\be\label{def_event_en}
\sum_{y\in NJ}\widetilde{\xi}^{\rho_0}_{Ns}(y)\leq
\left
\lfloor Nh\left(\rho_0+\frac{\varepsilon}{4}\right)\right\rfloor
\ee
for all $J\in\mathcal I$.
 Since (cf. \eqref{def_xitilde}) $\widetilde{\xi}^{\rho_0}\leq\xi^{\rho_0}$, 
by the law of large numbers  (recall that
${\xi}^{\rho_0}\sim\mu^{\alpha,\rho_0}$),  the event $E_N$ 
has probability tending to $1$ as $N\to+\infty$. 
Besides, by definition of  $\rho_0,\rho_1$  (see Proposition \ref{prop_mean_disc})
 and $\varepsilon$ (see \eqref{choose_m}), 
  we have $\rho(t_0,u')>\rho_0+2\varepsilon$ for 
every $u'\in[u-\Delta,u+\Delta]$. 
Let $F_N$ denote the event that 
\be\label{def_event_fn}
\sum_{y\in N J}{\eta}^N_{Ns}(y)>\lfloor Nh(\rho_0+3\varepsilon/2)\rfloor
\ee
for all $J\in\mathcal I$.
By Theorem 
\ref{th_hydro},
and by definition of  $\rho_0,\rho_1$,
 $\varepsilon$,  
 the event $F_N$ 
has probability tending to one as $N\to+\infty$. 
In the sequel we assume that the event $E_N\cap F_N$ holds. 
We set, for $J\in\mathcal I$,
\begin{eqnarray*}
 \mathcal N^{4,N,J}_{Ns} =\mathcal N^{4,N}_{Ns} 
  & := &  \sum_{y\in\Z\cap N J}{\eta}^{N}_{Ns}(y)
-\lfloor Nh(\rho_0+3\varepsilon/2)\rfloor\\
\mathcal N^{3,N,J}_{Ns} = \mathcal N^{3,N}_{Ns} 
 & = & \mathcal N^{2,N,J}_{Ns}
=  \mathcal N^{2,N}_{Ns} :=  \lfloor Nh\varepsilon/2\rfloor
\end{eqnarray*}
Since we are on  $F_N$  we have $\mathcal N^{4,N}_{Ns}\geq 0$, and because we are on  $E_N$,  we have
\[
\sum_{j=2}^4 \mathcal N^{j,N}_{Ns}  \leq 
\left[
\sum_{y\in\Z\cap NJ}
{\eta}^N_{Ns}(y)-
\sum_{y\in\Z\cap NJ}
\widetilde{\xi}^{\rho_0}_{Ns}(y)
\right]
\]
But by convexity of the function $u\mapsto u^+:=\max(u,0)$, 
for $J\in\mathcal I$, the above quantity is bounded above by
\be\label{def_totaldisc}
P_{J}:=\sum_{y\in\Z\cap NJ}\widetilde{\beta}^N_{Ns}(y)
\ee
which represents the total number of $({\eta}-\widetilde{\xi}^{\rho_0})$ discrepancies 
in the interval $\Z\cap NJ$ at time $Ns$.
Thus we have
\be\label{bound_disc}
\sum_{j=2}^4 \mathcal N^{j,N}_{Ns}\leq P_{J}
\ee
Thanks to \eqref{bound_disc}, among the above  number $P_{J}$ of  
$({\eta}-\widetilde{\xi} ^{\rho_0})$ discrepancies, 
for each $j\in\{2,3,4\}$, we may select a subset of  $\mathcal N^{j,N}_{Ns}$ discrepancies 
that we designate as  ${\eta}$ particles  of class $j$ at time $Ns$. 
These subsets are chosen 
arbitrarily provided they are disjoint (a given particle has a unique class). 
The remaining  ${\eta}$ particles 
not belonging to this selection (which contains some more 
 $({\eta}-\widetilde{\xi}^{\rho_0})$  discrepancies 
if inequality \eqref{bound_disc} is strict) will be first class  ${\eta}$ particles. 
We then denote, for $i\in\{1,2,3,4\}$,  ${\eta}^{i,N}_{Ns}$ as the configuration of 
${\eta}$ particles of classes $1$ to $i$. 
In a more formal way, the above construction means that we define configurations 
$\widetilde{\beta}^{j,N}_{Ns}\in\N^\Z$ for $j\in\{2,3,4\}$, in such a way that
\[
\sum_{j=2}^4\widetilde{\beta}^{j,N}_{Ns}\leq\widetilde{\beta}^N_{Ns}
\]
with, for each $J\in\mathcal I$,
\[
 \sum_{y\in \Z\cap NJ} \widetilde{\beta}^{j,N}_{Ns}(y)=\mathcal N^{j,N}_{Ns}
\]
We then set for $i\in\{1,2,3,4\}$,
\[
{\eta}^{i,N}_{Ns}:={\eta}^N_{Ns}-\sum_{j=i+1}^4 \widetilde{\beta}^{j,N}_{Ns}
\]
where the sum, which  
represents the configuration of particles of classes from $i+1$ to $4$,
is interpreted as empty for $i=4$.\\ \\
For $i\in\{1,2,3\}$, the configuration ${\eta}^{i,N}_{Ns}$ has a number of particles in each
box $\Z\cap (NJ)$ that is 
of order $Nh\rho^i$,  where $\rho^i$ is defined in \eqref{def:rho^i_1}. 
\end{proof}
\mbox{}\\ \\
The proof of Lemma \ref{lemma_mut1} relies on the following lemma  (using our current estimate in Corollary \ref{corollary_current}),  which states that 
the current in the ${\eta}^i$ system is close to the corresponding equilibrium current $f(\rho^i)$.
\begin{lemma}\label{lemma_current_eta}
Let $\theta\mapsto x_\theta$ be a $\Z$-valued path such that 
$\lim_{\theta\to+\infty}\theta^{-1}x_\theta=:v\in[0,V]$, and let $X_\theta^N:=x_{N\theta}$.
Let $\tau_0>0$,  $u_0\in[u-\Delta,u+\Delta]$ and 
$0<\tau<\frac{1}{V}\min\left(u_0-(u-\Delta),u+\Delta-u_0\right)$.
Then 
\be\label{eq_current_eta}
\lim_{N\to+\infty}\Prob\left(
H_N(u_0,\tau_0,\tau)
\right)=0
\ee
where $H_N(u_0,\tau_0,\tau)$ is the event that, for some $i\in\{1,2,3\}$,
\be\label{def_HN}
\left|
N^{-1}\Gamma^\alpha_{\lfloor Nu_0+X^N_.\rfloor}(N\tau_0,N(\tau_0+\tau),{\eta}^{i,N}_{N\tau_0})-\tau(f(\rho^i)-v\rho^i)
\right|\geq 2h(\rho_0+2\varepsilon)
\ee
\end{lemma}
\begin{proof}{Lemma}{lemma_current_eta}
 By Lemma \ref{current_critical}, the limit
\be\label{use_eq_current}
\lim_{N\to+\infty}N^{-1}\Gamma^\alpha_{\lfloor Nu_0+X^N_.\rfloor}(N\tau_0,\xi^{\rho^i}_{\tau_0})=f(\rho^i)
\ee
holds in probability.
Next, we apply Corollary \ref{corollary_current} to ${\eta}^{i,N}_{N\tau_0}$ 
and ${\xi}^{\rho^i}_{N\tau_0}$, and then  (exchanging their roles) to ${\xi}^{\rho^i}_{N\tau_0}$ and ${\eta}^{i,N}_{N\tau_0}$. Using \eqref{density_classes} from Lemma \ref{lemma_classes},
and the law of large numbers for $\xi^{\rho^i}$,
we find that the supremum on the r.h.s. of \eqref{current_comparison_2} is $o(N)+N\rho^i$, where
$o(N)$ denotes a random variable such that $o(N)/N$ vanishes in probability. Indeed, 
when evaluating the quantity
\be\label{difference_F}
F_{\lfloor Nu_0\rfloor}(x,{\eta}^{i,N}_{N\tau_0})-F_{\lfloor Nu_0\rfloor}(x,\xi^{\rho^i}_{N\tau_0}),
\ee
we decompose the interval between $\lfloor Nu_0\rfloor$ and $x$ into inner subintervals $I\in\mathcal I$ (see Lemma \ref{lemma_classes}) of length $h$ contained in it, and two possibly overlapping intervals of length $h$ at the boundaries.
On each inner subinterval, by \eqref{density_classes}, we have
\[
\sum_{z\in N\,I}\left[
{\eta}^{i,N}(z)-\xi^{\rho^i}(z)
\right]=o(N)
\]
On the other hand, the contribution of each boundary subinterval to the difference \eqref{difference_F} cannot exceed $h[o(N)+N\rho^i]$.
This implies \eqref{eq_current_eta}. 
\end{proof}
\mbox{}\\ 
\begin{proof}{Lemma}{lemma_mut1}
By Lemma \ref{lemma_current_eta}, we may assume that we are on the complement of $H_N=H_N(u_0,\tau_0,\tau)$.
We need to define speeds $v'_2$ and $v'_3$ satisfying  (i)  $v_3<v'_3<v'_2<v_2$,
 (ii)  $u_0+v'_2 \tau>u_0+3L+v'_3\tau$,
that is equivalent to
\be\label{vprime_bis}
v'_2-v'_3>\frac{3L}{\tau},
\ee
and  (iii) 
\begin{eqnarray}\label{condition_vprime2}
\tau\left[
f(\rho^2)-f(\rho^1)-v'_2(\rho^2-\rho^1)
\right]-4h(\rho_0+2\varepsilon)& \geq & 0\\
\label{condition_vprime3}
\tau\left[
f(\rho^3)-f(\rho^2)-v'_3(\rho^3-\rho^2)
\right]+4h(\rho_0+2\varepsilon)& \leq & 0
\end{eqnarray}
We have that \eqref{condition_vprime2}--\eqref{condition_vprime3} is equivalent to 
\be\label{vprime23_bis}
\min(v_2-v'_2,v'_3-v_3)
\geq
\frac{8h}{\tau\varepsilon}(\rho_0+2\varepsilon)
\ee
which is automatically satisfied by \eqref{def_vprime2}--\eqref{def_vprime3}.
We will check below that, given condition \eqref{cond_ctilde_h}, this choice of $v'_2$ and $v'_3$ also satisfies \eqref{vprime_bis}.
Having checked that $v'_2$ and $v'_3$ satisfy \eqref{vprime_bis} and \eqref{vprime23_bis},
the result follows from Lemma \ref{lemma_current_eta} and \eqref{current}.
Indeed, on the complement of $H_N$, the l.h.s. of \eqref{condition_vprime2} is a lower bound for the 
macroscopic flux of second class particles through a  path 
starting from  $Nu_0$  and moving with asymptotic speed $v'_2$.
While the l.h.s. of \eqref{condition_vprime3} is an upper bound for the 
macroscopic flux of second class particles through a path 
starting from  $N(u_0+3L)$  and moving with asymptotic speed $v'_2$.
\end{proof}
\mbox{}\\ \\\
\begin{proof}{Lemma}{lemma_mut2}
We must check that $v'_2$ and $v'_3$ defined by \eqref{def_vprime2}--\eqref{def_vprime3}
do satisfy \eqref{vprime_bis}. Indeed, \eqref{cond_ctilde_h} is equivalent to
\[
C'\varepsilon>6L/\tau,\quad
\frac{16h}{\tau\varepsilon}(\rho_0+2\varepsilon)\leq\frac{C'}{2}\varepsilon
\]
Thus
\begin{eqnarray*}
v'_2-v'_3 & = &  (v_2-v_3)-\frac{16h}{\tau\varepsilon}(\rho_0+2\varepsilon)\\
& \geq & C'\varepsilon-\frac{C'}{2}{\varepsilon}
\geq \frac{C'}{2}\varepsilon\geq\frac{3L}{\tau}
\end{eqnarray*}
\end{proof}
\begin{proof}{Corollary}{cor_mut1}
 Consider the set of second class ${\eta}$ particles initially (that is at time $N\tau_0$) in $[Nu_0,N(u_0+L)]$ which never coalesced in the time interval $[N\tau_0,N(\tau_0+\tau)]$,
the set of third class ${\eta}$ particles initially in $[N(u_0+2L),N(u_0+3L)]$ which never coalesced in this time interval $[N\tau_0,N(\tau_0+\tau)]$, and 
the set of  $\widetilde{\gamma}$ particles initially in $[N(u_0+L),N(u_0+2L)]$ which never coalesced in this time interval $[N\tau_0,N(\tau_0+\tau)]$. 
Let $E_N$ denote the event that these sets are all nonempty. Our goal is to show that $E_N$ has vanishing probability as $N\to+\infty$.
By Lemmas \ref{lemma_mut1} and \ref{lemma_mut2} (see explanations following the statement of Lemma \ref{lemma_mut1}), on an event $F_N$ with probability tending to $1$ as $N\to+\infty$,
any particle in the third set has either been crossed by all those of the first set, or by all those 
of the second one. Let us place ourselves on $E_N\cap F_N$. Suppose a particle from the third set finds itself during our time interval at the same site as a particle from the first set. If the ${\eta}$ particle has just jumped, by rule (II) of Subsection \ref{sec:class},
it coalesces with one of the uncoalesced $\widetilde{\xi}$ particles there, which is impossible. If the $\widetilde{\xi}$ particle has just jumped, by rule (I), it coalesces with one of the uncoalesced ${\eta}$ particles there, which is impossible. A similar argument shows that a particle from the third set cannot have been crossed by one from the second set. Hence $E_N\cap F_N=\emptyset$, which implies the desired result.
\end{proof}
\section{Proof of Theorem \ref{th_strong_loc_eq_super}}\label{sec:proof_loc_eq_2}
 We will prove \eqref{eq:loc_eq_super}, and combine it with 
 a lower bound derived from \eqref{coupling_2} to obtain \eqref{eq:loc_eq_super_bis}. 
The proof of \eqref{eq:loc_eq_super} involves comparison with a finite system limited by two sources.
To this end, Subsections \ref{subsec:jackson} and \ref{subsec:truncation}
contain preliminary material about invariant measures for disordered finite systems. In the latter two subsections, a generic environment
will be denoted by $\gendis$ to avoid confusion with the environment $\alpha$ of Theorem \ref{th_strong_loc_eq_super}.
\subsection{Invariant measure for the finite process with boundaries}\label{subsec:jackson}
Let 
$\gendis\in{\bf A}$, 
$l<0<r$, and 
$\overline{\eta}_0\in\overline{\mathbf X}$ 
a 
configuration  such that $\overline{\eta}_0(x)\in\N$ for 
$x\in(l,r)$, and
\be\label{infinite_ends}
\overline{\eta}_0(l)=\overline{\eta}_0(r)=+\infty
\ee
Consider the process $(\overline{\eta}^{\gendis}_t)_{t\geq 0}$, with initial configuration 
$\overline{\eta}_0$, and 
generator $L^{\gendis}$ (see \eqref{generator}). The restriction 
$(\eta_t^{\gendis,l,r})_{t\geq 0}$ of $(\overline{\eta}_t^{\gendis})_{t\geq 0}$ to 
$(l,r)$ is a Markov process on $\N^{(l,r)}$ with generator given by,  
for an arbitrary (since $g$ is bounded)  function $f$ on  $\N^{(l,r)}$,
\begin{eqnarray}\nonumber
L^{\gendis,l,r}f(\eta)&=&\sum_{x=l+1}^{r-2}
p\gendis(x)g(\eta(x))[f(\eta^{x,x+1})-f(\eta)]\\
\nonumber & + & \sum_{x=l+1}^{r-2} q\gendis(x+1)g(\eta(x+1))[f(\eta^{x+1,x})-f(\eta)]\\
\nonumber & + & q\gendis(l+1)g(\eta(l+1))[f(\eta-\delta_{l+1})-f(\eta)] \\ \nonumber
\nonumber & + & p\gendis(r-1)g(\eta(r-1))[f(\eta-\delta_{r-1})-f(\eta)]\\ \nonumber 
\nonumber &+ & p\gendis(l)[f(\eta+\delta_{l+1})-f(\eta)]\\
\label{gen_open} & + & q\gendis(r)[f(\eta+\delta_{r-1})-f(\eta)]
\end{eqnarray} 
Indeed, the effect of \eqref{infinite_ends} is to turn sites $l$ and $r$ 
into sources/sinks and isolate (since jumps are nearest-neighbour) the space interval 
$\Z\cap(l,r)$ from the part of the system located to the left of $l$ and right of $r$. 
{}From the point of view of the finite system, the source at site $l$ then behaves 
as a reservoir creating a particle at $l+1$ at rate $p\gendis(l)g(+\infty)=p\gendis(l)$, 
that is the rate of a particle jump from $l$ to $l+1$ in the infinite system with 
generator \eqref{generator}. A particle jumping from $l+1$ to $l$ in the infinite system 
is viewed in the finite system as a particle annihilation at site $l+1$ with the same rate
$q\gendis(l+1)g(\eta(l+1))$. Similar equivalences hold at the right end of the finite system.\\ \\
The above process is an open Jackson network, whose invariant measure is 
well-known in queuing theory. In our case this measure is explicit:
\begin{lemma}\label{lem-mu_c'} (\cite[Lemma 4.1]{bmrs1}) 
Set, for $x\in [l,r] \cap\Z$,
\be\label{def_lambda_inv}
\lambda^{\gendis,l,r}(x):=\frac{\gendis(r)-\gendis(l)}{1-\left(\frac{q}{p}\right)^{r-l}}
\left(\frac{q}{p}\right)^{r-x}+
\frac{\gendis(l)-\gendis(r)\left(\frac{q}{p}\right)^{r-l}}{1-\left(\frac{q}{p}\right)^{r-l}}
\in[\gendis(l),\gendis(r)]
\ee
The process with generator \eqref{gen_open} is positive recurrent if and only if 
\be\label{cond_rec}\lambda^{\gendis,l,r}(x)<\gendis(x),\quad\forall x\in(l,r)\cap\Z\ee
If condition \eqref{cond_rec} is satisfied, the unique invariant measure of 
the process is the product measure 
$\mu^{\gendis,l,r}$ on $\N^{(l,r)\cap\Z}$ with marginal 
$\theta_{\lambda^{ \kappa,l,r }(x)/\gendis(x)}$ at site $x\in(l,r)\cap\Z$,  that is
\be\label{def_inv_open}
\mu^{\gendis,l,r}(d\eta):=\prod_{x\in\Z\cap(l,r)}\theta_{\lambda^{\kappa,l,r}(x)/\gendis(x)}(d\eta(x))
\ee
\end{lemma}
\begin{remark}\label{remark_lemma_mu_c'}
If $\gendis(r)=\gendis(l)$, then the function $\lambda^{\gendis,l,r}$ defined by 
\eqref{def_lambda_inv} is constant and equal to this common value.
\end{remark}
\subsection{Truncation procedure}\label{subsec:truncation}
{}From now on, we consider the sequence $(\alpha^N)_{N\in\N}$ defined by
\be\label{shift_env}\alpha^N(.):=\tau_{x_N}\alpha(.)\ee
where $\alpha$ is the environment given in  Theorem \ref{th_strong_loc_eq_super}. 
Recall from \eqref{set_env} that $\alpha(.)$ lives on ${\bf A}=[0,1]^\Z$, that is
compact with respect to the product topology. Therefore it is
 enough to prove that the desired limit holds along any subsequence of values of $N\to+\infty$  
along which $\alpha^N$ converges to some 
$\overline{\alpha}\in\bf A$. 
Since $(x_N)_{N\in\N}$ is assumed to be a typical sequence,
$\overline{\alpha}$ satisfies the requirements of Definition \ref{def_typical}.  In particular (see point (i) of this definition), $\overline\alpha\in{\bf B}$. 
In the sequel, when writing $N\to+\infty$, it will be implicit that we restrict to such a subsequence. \\ \\
We need to introduce the following definitions.
Let $\kappa\in{\bf B}$ be an arbitrary environment. For $\varepsilon>0$, we define 
\begin{eqnarray}\label{def:ell}
A_\varepsilon(\gendis)&:=&\sup\{x\le 0:\gendis(x) \leq c+\varepsilon\}\in \Z^-\cup\{-\infty\}\,,\\
a_\varepsilon(\gendis)&:=&\inf\{x\ge 0:\gendis(x) \leq c+\varepsilon\}\in\overline{\N}\,.\label{def:err}
\end{eqnarray} 
with the usual conventions $\inf\emptyset=+\infty$ and $\sup\emptyset=-\infty$. It follows from these definitions that
\be\label{limit_a_epsilon}
\lim_{\varepsilon\to 0}A_\varepsilon(\gendis)=-\infty,\quad
\lim_{\varepsilon\to 0}a_\varepsilon(\gendis)=+\infty
\ee
\begin{eqnarray}
\label{finite_a_min}
\liminf_{x\to-\infty}\gendis(x)=c\Rightarrow\forall \varepsilon>0,\,A_\varepsilon(\gendis)>-\infty\\
\label{finite_a_plus}
\liminf_{x\to+\infty}\gendis(x)=c\Rightarrow\forall \varepsilon>0,\,a_\varepsilon(\gendis)<+\infty
\end{eqnarray}
 Recall that the environment $\alpha$ in Theorem \ref{th_strong_loc_eq_super} satisfies 
\eqref{not_too_sparse}. Thus \eqref{finite_a_min}--\eqref{finite_a_plus} imply that
$A_\varepsilon(\tau_n\alpha)$ 
and $a_\varepsilon(\tau_n\alpha)$ are finite for any $n\in\Z$. By condition (ii) of Definition \ref{def_typical}, the same holds for $A_\varepsilon(\overline\alpha)$ and $a_\varepsilon(\overline\alpha)$.
Besides,  as
 a consequence of \eqref{assumption_afgl}, we have:
\begin{lemma}\label{lemma_ergodic}  \cite[Lemma 4.4]{bmrs3} 
For every $\varepsilon>0$,
\be\label{limit_ergodic}
\lim_{n\to\pm\infty}n^{-1}a_\varepsilon(\tau_n\alpha)=0,
\lim_{n\to\pm\infty}n^{-1}A_\varepsilon(\tau_n\alpha)=0
\ee
\end{lemma}
 Let
\be\label{first_lr}
l=A_\varepsilon(\overline{\alpha}),\quad r=a_\varepsilon(\overline{\alpha})
\ee
By condition (ii) of Definition \ref{def_typical} and \eqref{finite_a_min}, we have 
$l>-\infty$, but we cannot exclude $r=+\infty$. However, 
{if} $r<+\infty$, the following holds.
On the one hand, \eqref{def_lambda_inv} implies 
$\lambda^{\overline{\alpha},l,r}(x)\in[\overline{\alpha}(l),\overline{\alpha}(r)]$ 
for all $x\in[l,r]$. On the other hand, by definitions \eqref{def:ell}--\eqref{def:err},
 $\max(\overline{\alpha}(l),\overline{\alpha}(r))\leq c+\varepsilon$ 
 and $\min\{\overline\alpha(x):x\in(l,r)\}>c+\varepsilon$. Hence, if $r<+\infty$,
\[
\lambda^{\overline{\alpha},l,r}(x)<\overline{\alpha}(x),\quad\forall x\in(l,r)
\]
 Thus the measure $\mu^{\overline{\alpha},l,r}$ is well defined and invariant 
 for the finite process in $(l,r)$ under environment $\overline{\alpha}$. The 
 idea to prove Theorem \ref{th_strong_loc_eq_super} is to show that this measure 
 provides an upper bound for $\eta^{\alpha,N}_{Nt}$ as $N\to+\infty$. However,
  we need to slightly modify the definitions of $l$ and $r$, to avoid the case 
  where $r=+\infty$, {and} to take into account that $\alpha^N$ is not exactly 
  $\overline{\alpha}$. \\ \\
We therefore replace $a_\varepsilon(\overline{\alpha})$ by $a'_\varepsilon(\overline{\alpha})$ 
defined as follows. Instead of \eqref{first_lr}, we set 
\be\label{let_us_set}
l=A_\varepsilon(\overline{\alpha}), \quad
r=a_\varepsilon(\overline\alpha){\bf 1}_{\left\{a_\varepsilon(\overline\alpha)<+\infty\right\}}+
\lfloor\varepsilon^{-1}\rfloor{\bf 1}_{\left\{a_\varepsilon(\overline\alpha)=+\infty\right\}}
\ee
and
\be\label{def_right_reservoir}
r':=a'_\varepsilon(\overline{\alpha}):=\left\{
\ba{lll}
a_\varepsilon(\overline{\alpha}) & \mbox{if} & a_\varepsilon(\overline{\alpha})<+\infty\\
\min\{ z\in\Z\cap(l, r):\,\lambda^{\overline{\alpha},l, r}(z)\geq\overline{\alpha}(z)\}
& \mbox{if} & a_\varepsilon(\overline{\alpha})=+\infty
\ea
\right.
\ee
with the convention that the  minimum is defined to be $r$ if the 
above set is empty. 
Let  $\widetilde{\alpha}$ be the environment defined by
\be\label{def_tildealpha}
\widetilde{\alpha}(x)=\overline{\alpha}(x)\mbox{ if }x\neq r',\quad
\widetilde{\alpha}(r')= \overline{\alpha}(r'){\bf 1}_{\{r'=r\}}+\lambda^{\overline{\alpha},l,r}(r')
{\bf 1}_{\{r'<r\}}
\ee
\begin{lemma}\label{lemma_new_trunc}
Let ${\alpha'}^N_\varepsilon$ be the environment defined by 
\be\label{def_gendisn}
{\alpha'}^N_\varepsilon(x)=\alpha^N(x)\mbox{ for }x\neq r', \quad
{\alpha'}^N_\varepsilon(r')=
\widetilde{\alpha}(r')+\delta_N
\ee
Then:  (i)  there exists a sequence $(\delta_N)_{N\in\N\setminus\{0\}}$ such that $\lim_{N\to+\infty}\delta_N=0$ and the following holds: for $N$ large enough, one has
\be\label{domination_boundary}
{\alpha'}^N_\varepsilon(r')>\alpha^N(r'),
\ee
and the process with generator
$L^{{\alpha'}^N_\varepsilon,l,r'}$ has the unique invariant measure 
\be\label{def_muepsilon_2}
\mu^{N}(\varepsilon):=\mu^{{\alpha'}^N_\varepsilon,l,r'},
\ee
 (ii) 
Let $S$ be a finite subset of $\Z^d$, and $\zeta\in\N^S$. Then 
\be\label{forsmall}S\subset(A_\varepsilon(\overline{\alpha}),a'_\varepsilon(\overline{\alpha})),\quad\mbox{for small enough } \varepsilon>0\ee
Besides,
\be\label{limit_muepsilon_0}
\lim_{N\to+\infty}\mu^{N}(\varepsilon)(\zeta)=
\mu^{\widetilde{\alpha},A_\varepsilon(\overline{\alpha}),a'_\varepsilon(\overline{\alpha})}(\zeta)
\ee
\be\label{limit_muepsilon}
\lim_{\varepsilon\to 0}\mu^{\widetilde{\alpha},A_\varepsilon(\overline{\alpha}),a'_\varepsilon(\overline{\alpha})}(\zeta)=\mu^{\overline{\alpha}}_c(\zeta)
\ee
\end{lemma}
The proof of Lemma \ref{lemma_new_trunc} uses the following identity.
\begin{lemma}\label{lemma_bartilde}
It holds that
\be\label{bartilde}
\lambda^{\widetilde{\alpha},l,r'}(x)=\lambda^{\overline{\alpha},l,r}(x),\quad\forall x\in[l,r']\cap\Z
\ee
\end{lemma}
\begin{proof}{Lemma}{lemma_bartilde}
If $r'=r$, this follows directly from \eqref{def_tildealpha}. 
If $r'<r$,  plug $\widetilde{\alpha}(r')=\lambda^{\overline{\alpha},l,r}(r')$ given by \eqref{def_lambda_inv} into $\lambda^{\widetilde{\alpha},l,r'}(x)$ defined again by \eqref{def_lambda_inv}.
\end{proof}
\mbox{}\\
\begin{proof}{Lemma}{lemma_new_trunc}
To prove (i) we must show that for large enough $N$, condition \eqref{cond_rec} is satisfied in our setting, that is 
\be\label{our_setting}
\lambda^{{\alpha'}^N_\varepsilon,l,r'}(x)<{\alpha'}^N_{\varepsilon}(x)=\alpha^N(x),\quad
\forall x\in(l,r')\cap\Z
\ee
Since
\be\label{while}
{\alpha'}^{N}_\varepsilon(l)=\widetilde{\alpha}(l),\quad
\lim_{N\to+\infty}\alpha^{'N}_\varepsilon(r')=\widetilde{\alpha}(r')
\ee
and $\lambda^{\kappa,l,r}$ defined in \eqref{def_lambda_inv} depends continuously on $(\kappa(l),\kappa(r))$, we have
\be\label{convergence_lambda}
\lim_{N\to+\infty}\lambda^{{\alpha'}^N_\varepsilon,l,r'}(x)=\lambda^{\widetilde{\alpha},l,r'}(x),\quad\forall
x\in\Z\cap(l,r')
\ee
Lemma \ref{lemma_bartilde} and definition \eqref{def_right_reservoir} of $r'$ imply
\begin{eqnarray}\label{construction_1}
\lambda^{\widetilde{\alpha},l,r'}(x)=\lambda^{\overline{\alpha},l,r'}(x)& < & \overline{\alpha}(x)=\widetilde{\alpha}(x),\quad\forall x\in\Z\cap(l,r')
\end{eqnarray}
Since ${\alpha}^N\to\overline{\alpha}$, we have 
\be\label{conv_env} 
{\alpha'}^N_\varepsilon(x)=\alpha^N(x)\to\overline{\alpha}(x)=\widetilde{\alpha}(x), \quad\forall x\in\Z\cap(l,r'),
\ee
Thus \eqref{convergence_lambda}, \eqref{construction_1} and \eqref{conv_env} imply that \eqref{our_setting} holds for all $N$ large enough. Next, since $\widetilde{\alpha}(r')\geq\overline{\alpha}(r')$ by \eqref{def_right_reservoir}--\eqref{def_tildealpha},  ${\alpha'}^N_\varepsilon(r')\to\widetilde{\alpha}(r')$ by \eqref{def_gendisn}, and ${\alpha}^N(r')\to\overline{\alpha}(r')$, we can find a sequence $(\delta_N)_{N\in\N}$ such that \eqref{domination_boundary} holds for the same values of $N$.\\ \\
We now turn to (ii). First, \eqref{forsmall} follows if we show
\be\label{weknow}\lim_{\varepsilon\to 0}A_\varepsilon(\overline{\alpha})=-\infty,\ee
\be\label{limit_aprime}\lim_{\varepsilon\to 0}a'_\varepsilon(\overline{\alpha})=+\infty\ee
The limit \eqref{weknow}  follows from \eqref{limit_a_epsilon}. To prove \eqref{limit_aprime}, we remark that
by \eqref{weknow} and \eqref{def_lambda_inv},
\[
\lim_{\varepsilon\to 0}\lambda^{\widetilde{\alpha},A_\varepsilon(\overline{\alpha}),\lfloor\varepsilon^{-1}\rfloor}(x)=c,\quad\forall x\in\Z
\]
which implies \eqref{limit_aprime}, since $\alpha(x)>c$ for all $x\in\Z$.
Next, \eqref{limit_muepsilon_0} follows from \eqref{def_inv_open}, \eqref{while}, \eqref{convergence_lambda} and \eqref{conv_env}. 
Finally, \eqref{limit_muepsilon} follows from \eqref{def_inv_open} if we show that
\be\label{enough}
\lim_{\varepsilon\to 0}\lambda^{\widetilde{\alpha},A_\varepsilon(\overline{\alpha}),a'_\varepsilon(\overline{\alpha})}(x)=c,\quad\forall x\in\Z
\ee
But \eqref{enough} follows from \eqref{def_lambda_inv} applied to $\kappa=\widetilde\alpha$, $l=A_\varepsilon(\overline\alpha)$, $r=a'_\varepsilon(\overline\alpha)$, since by definition \eqref{def:ell} of $A_\varepsilon$ and condition (ii) of Definition \ref{def_typical}, we have
\[
\lim_{\varepsilon\to 0}\overline\alpha[A_\varepsilon(\overline\alpha)]=c
\]
\end{proof}
\subsection{Proof of  \eqref{eq:loc_eq_super}  and \eqref{eq:loc_eq_super_bis}}\label{subsec:main_proof_super}
We are now ready to prove the upper bound in Theorem \ref{th_strong_loc_eq_super}. 
Define the initial configuration 
$\eta'^N_0\in\overline{\mathbf X}$  by 
\[{\eta}'^N_0(z)=\left\{
\ba{lll}
\tau_{x_N}\eta_0^N(z) & \mbox{if}  & z\not\in\{l,r'\}\\
+\infty & \mbox{if} & z\in\{l,r'\}
\ea
\right.
\]
We consider the process with generator
 $L^{\alpha'^N_\varepsilon}$  and initial configuration 
$\eta'^N_0$,  where $\alpha'^N_\varepsilon$ is defined by \eqref{def_gendisn}. 
We denote this process by $(\eta'^N_t)_{t\geq 0}$. We denote by $(\overline{\eta}^N_t)_{t\geq 0}$
the process with initial configuration $\eta'^N_0$ and generator 
 $L^{\alpha^N}$.
Recall from the first paragraph of Subsection \ref{subsec:jackson} 
that the restrictions of these two processes to the space interval
$\Z\cap
(l,r')
$
are Markov processes on 
$\N^{\Z\cap
(l,r')
}$ 
with respective generators 
$L^{\alpha'^N_{\varepsilon},l,r'}$ and
$L^{\alpha^N,l,r'}$. 
By the truncation procedure described in Subsection 
\ref{subsec:truncation}, the latter has a unique stationary measure 
$\mu^N(\varepsilon)$ defined in \eqref{def_muepsilon_2}.
 By attractiveness, since  $\tau_{x_N}\eta_0^N\leq\eta'^N_0$ , 
 we have  $\tau_{x_N}\eta^N_t\leq\overline{\eta}^N_t$  for every $t\geq 0$. 
On the other hand, by \eqref{domination_boundary}, the entrance rate   
$(\eta'^N_t)_{t\geq 0}$ has been increased with respect to the entrance rate $q\alpha^N(r')$ of $(\overline{\eta}^N_t)_{t\geq  0 }$, 
thus  (see \cite[Proposition 3.1]{bmrs2}) 
$\overline{\eta}^N_t\leq\eta'^N_t$, hence  $\tau_{x_N}\eta^N_t\leq\eta'^N_t$. This implies 
\be\label{compare_finite_infinite}
\tau_{x_N}\eta^N_t\leq\eta'^N_t
\ee 
We couple the process $(\eta'^N_t)_{t\geq 0}$ (via Harris construction) to its stationary version, 
namely the process that we denote by $({\xi'}^N_t)_{t\geq 0}$, which has 
the same generator but initial distribution 
$\mu^N(\varepsilon)$. 
Without loss of generality (since we need to prove an upper bound for $\eta'^N_t$), 
we may assume that $\eta'^N_0\geq\xi'^N_0$, by replacing $\eta'^N_0$ with $\eta'^N_0\vee \xi'^N_0$.
It follows from attractiveness that $\eta'^N_t\geq\xi'^N_t$ for all $t>0$.
We consider a process $(X_t^N)_{t\geq 0}$ defined as follows. Initially, 
we give some arbitrary finite value to $X^N_0$ (for instance $0$).
Then, the evolution of $X_t^N$ is defined as follows.
Recall the Harris construction described in Subsection \ref{subsec:Harris} 
from the Poisson
measure $\omega$ in \eqref{def_omega}. We can apply this construction to our 
processes in the finite interval $(l,r)\cap\Z$
by viewing them as processes on $\Z$ with infinitely many particles at sites $l$ and $r$.\\ \\
{\em First case.}
an $\eta'$ particle alone jumps to the right. This occurs at time $t>0$ if 
$(t,x,u,1)\in\omega$  for some $(x,u)\in((l,r']\cap\Z)\times(0,1)$, and
\be\label{only_eta}
\alpha(x)g[{\xi'}^N_{t-}(x)]<u\leq\alpha(x)g[{\eta'}^N_{t-}(x)]
\ee
At such times $t$, we set $X^N_t=X^N_{t-}+1$.\\ \\
{\em Second case.}
an $\eta'$ particle alone jumps to the left. This occurs at time $t>0$ if 
$(t,x,u,-1)\in\omega$ for some $(x,u)\in([l,r')\cap\Z)\times(0,1)$,
and \eqref{only_eta} holds.
At such times $t$, we set $X^N_t=X^N_{t-}-1$.\\ \\
We may interpret the above evolution rule by saying that whenever 
an $(\eta'-\xi')$ discrepancy moves to the right (resp. left), $X_t^N$
increases (resp. decreases) by one unit, including the case where 
this discrepancy leaves the interval $(l,r)$ (in which case it never reenters). It follows 
that, at any time $t>0$, the maximum possible value for the variation
 $X_t^N-X_0^N$ corresponds to the case where all discrepancies were initially 
 at site $l+1$ and have moved to site $r'$ by time $t$, in which case $X^N_.$ 
 has been incremented $r'-l$ times for each discrepancy. Therefore,
\be\label{bound_current_discrep}
X_t^N-X_0^N\leq (r'-l)
\sum_{z=l+1}^{r'-1}[\eta'^N_0(z)-\xi'^N_0(z)]
\ee
and, for any $T>0$,
\be\label{current_discrep}
\Exp(X^N_T)-\Exp(X^N_0)=\sum_{z=l+1}^{r'-1}
(p-q)\alpha(z)\Exp\int_0^T\left[
g(\eta'^N_t(z))-g(\xi'^N_t(z))
\right]dt
\ee
We define
\[
T^N_\infty:=\inf\left\{
t>0:\,\eta'^N_t=\xi'^N_t
\right\}
\]
to be the time at which all $(\eta'-\xi')$ discrepancies have left our finite system. 
Since $(\xi'^N_t)_{t\geq 0}$ is a stationary ergodic process, and the empty
configuration has some positive probability $p_N$ with respect to the 
stationary measure 
$\mu^N(\varepsilon)$, 
by the ergodic theorem, we have
\be\label{fraction_of_time}
\lim_{T\to+\infty}\Prob\left(
\left|
t\in[0,T]:\,\xi'^N_t(.)\equiv 0
\right|\geq \frac{1}{2}p_N T
\right)
=1
\ee
Let us denote the event in \eqref{fraction_of_time} by $E^N_T$. 
We insert the indicators of $E^N_T$ and $\{T^N_\infty>T\}$ in \eqref{current_discrep} 
and observe that on $E^N_T$, the integral in \eqref{current_discrep} is bounded 
below by  $\frac{1}{2}p_N\,c\,g(1)\,T$.  Using also 
\eqref{bound_current_discrep},
we arrive at 
\begin{eqnarray}
& & (r'-l)
\sum_{z=l+1}^{r'}[\eta'^N_0(z)-\xi'^N_0(z)]  \nonumber\\
& \geq & (p-q)cg(1)\frac{1}{2}p_N T\Prob(T^N_\infty>T,E^N_T)\label{we_arrive_at}
\end{eqnarray}
Since the initial configuration $\eta^N_0$ is assumed to have a density profile in $L^\infty$, 
there exists a sequence $(\varepsilon_N)_{N\geq 1}$ tending to $0$ as $N\to+\infty$, such that
the l.h.s. of \eqref{we_arrive_at} is bounded by $N\varepsilon_N$.
Besides, by  \eqref{limit_muepsilon_0},  $p_N$ remains bounded away 
from $0$.  
{}From this and \eqref{fraction_of_time}--\eqref{we_arrive_at}, 
setting $t_N=\sqrt{N}\vee (N\varepsilon_N)$,  we have
$\lim_{N\to+\infty}t_N=+\infty$, $\lim_{N\to+\infty}t_N/N=0$, and
\be\label{from_this}
\lim_{N\to+\infty}\Prob(T_\infty^N>t_N)=0
\ee
We have thus shown that with probability tending to one, all $(\eta'-\xi')$ discrepancies 
have left the box in time $o(N)$. 
Thus for any $t>0$, recalling \eqref{compare_finite_infinite}, 
\begin{eqnarray}
&&\lim_{N\to+\infty}\Prob\left(
\tau_{x_N}\eta^N_{Nt}(z)\leq\eta'^N_{Nt}(z)\leq\xi'^N_{Nt}(z),\right.\nonumber\\
&&\left.\qquad\forall z\in\Z\cap[A_\varepsilon(\overline{\alpha})+1,a'_\varepsilon(\overline{\alpha})-1]
\right)=1\label{finite_infinite}
\end{eqnarray}
By definition, the distribution of 
$\xi'^N_t$ is $\mu^N(\varepsilon)$.  \\ \\
We can now conclude the proof of \eqref{eq:loc_eq_super}.
It is enough to consider a bounded nondecreasing test function $\psi$.
Let $S$ be a finite subset of $\Z$ containing the support of $\psi$. 
By \eqref{finite_infinite},
we have that for $\varepsilon>0$ small enough, 
\begin{eqnarray*}
\limsup_{N\to+\infty}\Exp \psi\left(\tau_{x_N}\eta^N_{Nt}\right) & \leq & \limsup_{N\to+\infty}\Exp \psi\left(\xi'^N_{Nt}\right)\\
& = & \limsup_{N\to+\infty}\int_{\mathbf X} \psi(\xi)d\mu^{N}(\varepsilon)(\xi)
\end{eqnarray*}
Now we  let $\varepsilon\to 0$,  and apply 
 \eqref{limit_muepsilon_0}--\eqref{limit_muepsilon} 
to the above equality. This yields
\[
\limsup_{N\to+\infty}\Exp \psi\left(\tau_{x_N}\eta^N_{Nt}\right)
\leq \int_{\mathbf X}\psi(\eta)d\mu_c^{\overline{\alpha}}(\eta)
\]
 We finally  prove \eqref{eq:loc_eq_super_bis}. To this end, 
we use \eqref{coupling_2} and argue similarly to the proof of 
Theorem \ref{th_strong_loc_eq}. Since $\rho_*(t,u)\geq\rho_c$, 
we can find $\Delta$ small enough and $\rho_0<\rho_1$, 
where $\rho_1$ is defined by \eqref{def:rho1}, such that $\rho_0$ is 
arbitrarily close to $\rho_c$.
We couple the three processes $\xi^{\alpha,\rho_0}_.$, $\eta^{\alpha,N}_.$ 
and $\xi'^{N}_.$. Let $G_N$
denote the event in \eqref{finite_infinite} and $F_N$ the one in 
\eqref{coupling_2}. Thus on $F_N\cap G_N$, we
have $\xi^{\alpha,\rho_0}_{Nt}\leq\tau_{x_N}\eta^{\alpha,N}_{Nt}\leq \xi'^{\alpha}_{Nt}$. 
Similarly to \eqref{using_lipsi}, we can write
\begin{eqnarray}
&&\!\!\!\!\!\!\!\!\!\!\!\!\!\!\!\!\!\!\left|
\Exp\psi\left(\tau_{x_N}\eta^{\alpha,N}_{Nt}\right)
-\int_{\mathbf X}\psi(\eta)d\mu^{\tau_{x_N}\alpha,\rho_c}(\eta)
\right|\nonumber\\ &\leq  & \left|
\Exp\psi\left(\tau_{x_N}\eta^{\alpha,N}_{Nt}\right)-\Exp\psi\left(\tau_{x_N}\xi^{\alpha,\rho_0}_{Nt}\right)
\right|\nonumber\\
&&+ \left|
\Exp\psi\left(\tau_{x_N}\xi^{\alpha,\rho_0}_{Nt}\right)
-\int_{\mathbf X}\psi(\eta)d\mu^{\tau_{x_N}\alpha,\rho_c}(\eta)
\right|  \label{using_lipsi_again}
\end{eqnarray}
Proceeding as in \eqref{inequality_follows}--\eqref{forthelast}, we  bound 
the second term on the r.h.s. of \eqref{using_lipsi_again} by 
\begin{eqnarray}
&&\!\!\!\!\!\!\!\!\!\!\!\!\!\!\!\!\!\! 2||\psi||_\infty\Exp
\sum_{z\in S}\left[\xi^{\alpha,\rho_c}_{0}( x_N+ z)-\xi^{\alpha,\rho_0}_0( x_N+ z)
\right]\nonumber\\
& = &
2||\psi||_\infty
\sum_{z\in S}\left\{
R\left(
\frac{c}{\alpha(x_N+z)}
\right)-R\left(
\frac{\overline{R}^{-1}(\rho_0)}{\alpha(x_N+z)}
\right)
\right\}\label{proceeding_as_in}
\end{eqnarray}
By condition (i) of Definition \ref{def_typical}, $\overline{\alpha}(z)<c$ 
for every $z\in\Z$. Thus the quantity in \eqref{proceeding_as_in} converges as $N\to+\infty$ to
\be\label{proceeding_converges}
2||\psi||_\infty
\sum_{z\in S}\left\{
R\left(
\frac{c}{\overline{\alpha}(z)}
\right)-R\left(
\frac{\overline{R}^{-1}(\rho_0)}{\overline{\alpha}(z)}
\right)
\right\}
\ee
and the latter vanishes as $\rho_0\to\rho_c$.\\ \\
For the first term on the r.h.s. of \eqref{using_lipsi_again}, 
we insert the indicator of $F_N\cap G_N$.
Since on this event we have $\tau_{x_N}\xi^{\alpha,\rho_0}\tau_{x_N}\leq\eta^{N}_{Nt}\leq\xi'^{N}_{Nt}$,
 using \eqref{lipsi}, we obtain the upper bound (recalling that $\xi'^{N}_.$ is stationary 
 with distribution given by \eqref{def_muepsilon_2})
\begin{eqnarray}
&&\!\!\!\!\!\!\!\!\!\!\!\!\!\!\!\!\!\! 2||\psi||_\infty\sum_{z\in S}\left[
\Exp\xi'^{N}_{Nt}(z)-\Exp\xi^{\alpha,\rho_0}(x_N+z)
\right]\nonumber\\
& = & 2||\psi||_\infty\sum_{z\in S}\left[
\int_{\mathbf X}\xi(x)d\mu^N(\varepsilon)(\xi)-R\left(
\frac{\overline{R}^{-1}(\rho_0)}{\alpha(x_N+z)}
\right)
\right]
\label{upper_bound_as_before}
\end{eqnarray}
Now we let $N\to+\infty$ and $\varepsilon\to 0$ in \eqref{upper_bound_as_before}, 
and use \eqref{limit_muepsilon_0}--\eqref{limit_muepsilon}. The quantity 
\eqref{upper_bound_as_before} then converges to
\eqref{proceeding_converges}.
This concludes the proof of \eqref{eq:loc_eq_super_bis}. 
\section{Proof of Theorem \ref{prop_strong_loc_eq_super_cesaro}}\label{sec:cesaro}
We  define the shifted environment $\alpha^N$ as in \eqref{shift_env}. 
As in the proof of Theorem \ref{th_strong_loc_eq_super}, we consider 
a subsequence of $(x_N)_{N\in\N}$ along which $\alpha^N$ has a limit 
$\overline{\alpha}$. However, we no longer make any assumption on this 
limiting $\overline{\alpha}$. It will be enough to establish \eqref{eq:loc_eq_super_cesaro} 
along any such subsequence, that we shall still denote 
 by $(x_N)$ for notational simplicity. Let us set
\be\label{def_nutdelta}
\nu^{N,\delta}_t(.):=\frac{1}{N\delta}  \int^{Nt}_{N(t-\delta)} \Prob(\tau_{x_{N}} \eta^{\alpha,N}_{s}\in \cdot ) ds
\ee
so that the expectation in \eqref{eq:loc_eq_super_cesaro} can be rewritten as
$\int \psi(\eta)d\nu^{N, \delta }_t(\eta)$. Our problem is thus to show that 
in the double limit $N\to+\infty$ followed by $\delta\to 0$, $\nu_t^{N,\delta}$
converges weakly to $ \mu^{\overline{\alpha}}_c$. 
To this end, the main step is to show that $\nu$ satisfies certain properties which we will show to characterize $\mu^{\overline{\alpha}}_c$:
\begin{proposition}\label{prop:inv_env}
Any subsequential weak limit $\nu$ of the sequence $(\nu^{N,\delta}_t)_{N\in\N}$ 
as $N \to+ \infty$ satisfies the following properties:\\ \\
(i) $\nu$ is an invariant measure for the Markov process 
with generator $L^{\overline{\alpha}}$ defined by \eqref{generator}, 
that is the zero-range process in environment $\overline{\alpha}$.\\ \\
(ii) $\nu\geq\mu^{\overline{\alpha}}_c$.\\ \\
(iii) $\dsp \overline{\alpha}(x)\int_{\overline{\mathbf X}}g(\eta(x))d\nu(\eta)=c$ for every $x\in\Z$.
\end{proposition}
We now derive Theorem \ref{prop_strong_loc_eq_super_cesaro} 
from Proposition \ref{prop:inv_env}, which will be proved at the end of this section.  
Since we also have 
\[
\dsp \overline{\alpha}(x)\int_{\overline{\mathbf X}}g(\eta(x))d\mu^{\overline{\alpha}}_c(\eta)=c,
\]
if $g$ were strictly increasing, properties (ii) and (iii) of Proposition \ref{prop:inv_env} 
would directly imply $\nu=\mu^{\overline{\alpha}}_c$ using a coupling argument. In the general 
case where $g$ is only assumed nondecreasing, to make this coupling still effective, 
we need an additional lemma.\\ \\
For $x\in\Z$ and $t\geq 0$, let $\varphi_{x,0}$ denote the function defined on $\bf X$ by $\varphi_{x,0}(\eta)=g(\eta(x))$, and $\varphi^{\overline{\alpha}}_{x,t}=P_t^{\overline{\alpha}}\varphi_{x,0}$, that is
\be\label{def_varphixt}
\varphi_{x,t}^{\overline{\alpha}}(\eta):=\Exp_{\eta}g\left(\eta^{\overline{\alpha}}_t(x)\right)
\ee
where $\left(\eta^{\overline{\alpha}}_t\right)_{t\geq 0}$ denotes a process 
with generator $L^{\overline{\alpha}}$, cf. \eqref{generator}, and index $\eta$ 
means that this process has initial configuration $\eta$.
The function $\varphi^{\overline{\alpha}}_{x,t}$ is nondecreasing because the semigroup 
$(P_t^{\overline{\alpha}})_{t\geq 0}$ is monotone, but in fact we have 
the following stronger statement.
\begin{lemma}\label{lemma_slight_problem}
Let $(\eta,\xi)\in\overline{\mathbf X}^2$ and $x\in\Z$ such that $\eta\leq\xi$ and $\eta(x)<\xi(x)$.
Then, for every $t>0$, $\varphi_{x,t}^{\overline{\alpha}}(\eta)<\varphi^{\overline{\alpha}}_{x,t}(\xi)$.
\end{lemma}
\begin{proof}{Lemma}{lemma_slight_problem}
By attractiveness, we have $\eta^{\overline{\alpha}}_t(x)\leq\xi^{\overline{\alpha}}_t(x)$, 
and since $g$ is nondecreasing,  
$g\left(\eta^{\overline{\alpha}}_t(x)\right)\leq g\left(\xi^{\overline{\alpha}}_t(x)\right)$.
By \eqref{def_varphixt}, it is thus enough to show that 
\be\label{order_g}
g\left(\eta^{\overline{\alpha}}_t(x)\right)< g\left(\xi^{\overline{\alpha}}_t(x)\right)
\ee
with positive probability. To this end, we can use the graphical construction to show
that, given $\eta$ and $\xi$, one can find an event of positive probability in the 
Poisson space $(\Omega,\mathcal F,\Prob)$ (see Subsection \ref{subsec:Harris}) on which 
\be\label{onwhich}
\eta^{\overline{\alpha}}_t(x)=0<\xi^{\overline{\alpha}}_t(x), 
\ee
which implies \eqref{order_g}, because $g(0)=0<g(1)$ (see \eqref{properties_g}).
Indeed, to achieve \eqref{onwhich}, it is enough to require that up to time $t$, 
no $\eta$ or $\xi$ particle 
has ever jumped to $x$ from any other site, and that a number $\eta(x)$ of coupled 
$(\eta,\xi)$ jumps have occurred, so that there are no more $\eta$-particles left 
at site $x$, but at least one $\xi$-particle.
\end{proof}
\mbox{}\\ \\
\begin{proof}{Theorem}{prop_strong_loc_eq_super_cesaro}
We must show that any subsequential weak limit $\nu$ (as defined in Proposition \ref{prop:inv_env}) of $(\nu^{N,\delta}_t)_{N\geq 1}$ is equal to $\mu^{\overline\alpha}_c$.
By (ii) of Proposition \ref{prop:inv_env}, since $g$ is nondecreasing,
\be\label{fluxoverc}
\overline\alpha(x)\int_{\overline{\mathbf X}}g(\eta(x))d\nu(\eta)\geq 
\overline\alpha(x)\int_{\overline{\mathbf X}}g(\eta(x))d\mu^{\overline{\alpha}}_c(\eta)= c,\quad\forall x\in\Z
\ee
It follows from 
(iii) of Proposition \ref{prop:inv_env} that
\be\label{fluxequalc}
\overline\alpha(x)\int_{\overline{\mathbf X}}g(\eta(x))d\nu(\eta)= c,\quad\forall x\in\Z
\ee
Since $\nu$ and $\mu^{\overline{\alpha}}_c$ are invariant measures for $L^{\overline{\alpha}}$ (see Proposition \ref{prop:inv_env}), for any $t>0$, by \eqref{fluxoverc} (recall \eqref{def_varphixt}),
\be\label{fluxequalc_bis}
\int_{\overline{\mathbf X}}g(\eta(x))d\nu(\eta)=\int_{\overline{\mathbf X}}\varphi^{\overline{\alpha}}_{x,t}(\eta)d\nu(\eta)=
\int_{\overline{\mathbf X}}\varphi^{\overline{\alpha}}_{x,t}(\eta)d\mu^{\overline{\alpha}}_c(\eta)=
\int_{\overline{\mathbf X}}g(\eta(x))d\mu^{\overline{\alpha}}_c(\eta)
\ee
By (ii) of Proposition \ref{prop:inv_env} and Strassen's theorem, there exists a coupling measure $\pi(d\eta,d\xi)$ with first marginal $\nu(d\eta)$ 
and second marginal $\mu^{\overline{\alpha}}_c(d\xi)$, such that
\be\label{coupling_munu}
\pi\left(
\left\{
(\eta,\xi)\in\overline{\mathbf X}^2:\,\eta\geq\xi
\right\}
\right)=1
\ee
 We also have, by \eqref{fluxoverc_bis}, for every $x\in\Z$,
\[
\int_{\overline{\mathbf X}^2}\varphi^{\overline{\alpha}}_{x,t}(\eta)d\pi(\eta,\xi)=\int_{\overline{\mathbf X}^2}\varphi^{\overline{\alpha}}_{x,t}(\xi)d\pi(\eta,\xi)
\]
The above equality, combined with \eqref{coupling_munu}, implies (since $\varphi^{\overline{\alpha}}_{x,t}$ is nondecreasing)
\be\label{above_implies}
\pi\left(
\left\{
(\eta,\xi)\in\overline{\mathbf X}^2:\,\varphi_{x,t}^{\overline\alpha}(\eta)=\varphi_{x,t}^{\overline\alpha}(\xi)
\right\}
\right)=1
\ee
In view of \eqref{coupling_munu} and Lemma \ref{lemma_slight_problem}, 
since \eqref{above_implies} holds for every $x\in\Z$, we obtain that 
\[
\pi\left(
\left\{
(\eta,\xi)\in\overline{\mathbf X}^2:\,\eta=\xi
\right\}
\right)=1
\]
which concludes the proof.
\end{proof}
\mbox{}\\ \\
We now turn to the proof of Proposition \ref{prop:inv_env}.
\mbox{}\\ \\
 The next two lemmas, which analyze the current for a zero-range process 
through a fixed site $x$ on a large time interval, are the main steps 
for the proof of statement (ii) in Proposition \ref{prop:inv_env}.
 The first lemma enables us to replace the current by a suitable compensator. 
\begin{lemma}
\label{lemswitch}
Given a family of zero range processes $(\zeta^N_s )_{s \geq 0 } $ 
for $N \in \Z_+$ with environments $\kappa^N\in{\bf A}$, suppose that $(t_N)_{N\in\N}$ is a sequence going 
to $+\infty$ such that, for all $x \in \Z$, 
the following convergences hold in probability:
\be\label{assumption_lemswitch}
\lim_{N\to+\infty}t_N^{-1}\zeta^N_0(x) = \ 0, 
\lim_{N\to+\infty}t_N^{-1}\zeta^N_{t_N}(x) = \ 0
\ee
Then for every $x\in\Z$,
\be\label{switch}
\lim_{N\to+\infty}\frac{1}{t_N}\left(\Gamma^{\kappa^N}_x (t_N, \zeta^N_0) - 
  \kappa^N(x)   \int _0 ^ {t_N} (p-q)g( \zeta^N_s(x))ds \right) 
  = 0
\ee
in probability.
\end{lemma}
\begin{proof}{Lemma}{lemswitch}
   The current
$\Gamma^{\kappa^N}_x(s,\eta) $ has compensator 
\[
A= \kappa^N(x) p \int _0 ^ s g( \eta_s(x))ds \ - \  \kappa^N(x+1) q \int _0 ^ s g( \eta_s(x+1))ds
\]
 If we let  $M = \Gamma^{\kappa^N}_x(s,\eta) -A$, since $\Gamma^{\kappa^N}_x(s,\eta)$ 
 is a Poisson point process with jump size one and $A$ is continuous it can be 
 shown that the quadratic variation of $M$ is $A$. 
Now given $g$ is bounded, Doob's inequality yields that
\be\label{badswitch}
{t_N}^{-1}\left(\Gamma^\alpha_x(t_N,\eta)
 -  \kappa^N(x) \int _0 ^ {t_N} pg( \eta_s(y))ds + \kappa^N(x+1) q \int _0 ^ {t_N} g( \eta_s(x+1))ds\right)
\ee
converges  to 0
in probability as $t_N$ tends to infinity.  However, \eqref{badswitch} is not adapted to our purpose. 
The proof of \eqref{switch} requires a more careful analysis.\\ \\
{}From our construction of the zero-range process, we can associate 
to each particle that leaves $x$ a ($p,q$) random walk that is killed upon its return to $x$ 
(which is certain if the particle leaves $x$ to the left).  
We can accordingly write, for any $t>0$,
\[
\Gamma^{\kappa^N}_x(t,\eta) \ ={\mathcal N} (x,t) \ + {\mathcal N}_1(x,t) \ - {\mathcal N}_2(x,t),
\]
where ${\mathcal N}(x,t)$ is the number of particle emissions from $x$ to the right 
where the associated random walk never returns to $x$, 
${\mathcal N}_1(x,t)$ is the number of emissions from $x$ to the right which will 
eventually return but which have not done so 
by time $t$ and ${\mathcal N}_2(x,t)$ is the number of particles to the right of $x$ at time $0$ 
which hit $x$ by time $t$.
These three random variables can be understood (and bounded) as follows.\\ \\
\indent
We claim that ${\mathcal N}(x,t)$ is a point process with random intensity 
$(p-q) \kappa^N(x)g(\zeta^N_s(x))$, 
thus with compensator  $\int _0 ^ t (p-q)\kappa^N(x)g( \zeta^N_s(x))ds $.
Indeed, if we first consider the process of particles leaving $x$ to the right 
(regardless of whether an emitted particle ever comes back to $x$), 
we obtain a point process with intensity $p\kappa^N(x)g(\zeta^N_s(x))$.
Then we associate to each such particle a mark (say  $0$ or $1$) to indicate 
whether the particle will ever come back to $x$ (in which case the mark value 
is $1$) or not. We note that this mark depends only on the skeleton of the particle's 
random walk. These skeletons (and thus the marks) are mutually independent, and the set
 of these skeletons is independent of the original point process.
Using gambler's  ruin estimate, the probability for each mark to be $0$ is $1-\frac{q}{p}$. 
Hence, as claimed, ${\mathcal N}(x,t)$ is a point process with intensity 
$(1-\frac{q}{p})p\kappa^N(x)g(\zeta^N_s(x))=(p-q)\kappa^N(x)g(\zeta^N_s(x))$.
Thus,
\be\label{martingale_marks}
{\mathcal N}(x,t)- \kappa^N(x) \int _0 ^ {t}(p-q)g( \zeta^N_s(x))ds   
\ee
is a martingale whose quadratic variation can be bounded by  (recall \eqref{ginfty}) 
\be\label{bound_variation}
g(+\infty)\kappa^N(x)t\leq t
\ee
Then ${\mathcal N}_1(x,t)$ can be bounded (for each positive integer $r$) by 
${\mathcal N}^r_1(y,t)  + {\mathcal N}^{r \prime}_1(y,t)$, 
where ${\mathcal N}^r_1(x,t)$ is the number of emissions before time $t$ 
from $x$ that hit site $x+r$ 
and then return to $x$  and ${\mathcal N}^{r \prime}_1(x,t)$ is the number 
of particles in $(x,x+r)$ at time 
$t$.
Finally, ${\mathcal N}_2(x,t)$ has the elementary bound (for each positive integer $r$)  
${\mathcal N}^r_2(x,0) \, + \, {\mathcal N}^{r\prime}_2(x,t)$, where 
${\mathcal N}^r_2(x,0)$ is simply the number of particles 
in interval $(x,x+r)$ at time $0$
and  $ {\mathcal N}^{r \prime}_2(x,t)$ is the number of emissions from $x+r$ by time $t$ 
that hit site $x$ before returning to $x+r$ (including those that do so after time $t$).  
Thus ${\mathcal N}^r_2(x,0)$ is a fixed finite random variable not depending on $t$. \\ \\
On the other hand, ${\mathcal N}^{r \prime}_2(x,t)$ and ${\mathcal N}^{r}_1(x,t)$
are stochastically bounded by a Poisson random variable of parameter 
$t g(\infty)(\frac{q}{p})^r$.
Indeed, since the hitting of $x$ by a particle emitted from $x+r$ depends only 
on the skeleton of the particle's random walk, 
we may argue (as above for $\mathcal N$) that ${\mathcal N}^{r \prime}_2(x,t)$ has intensity
$(\frac{q}{p})^r \kappa^N(x+r)g(\zeta^N_s(x+r))$, $(\frac{q}{p})^r$ 
being the probability that a particle starting from $x+r$ and performing 
a $(p,q)$-random walk ever hits $x$. \\ \\
Thus, given any $\varepsilon > 0 $, $r$ may be chosen so large that  
\be\label{solarge}
\Exp({\mathcal N}^r_1(x,t)) + \Exp({\mathcal N}^{r \prime}_2(x,t)) \ \leq \varepsilon  t
\ee
On the other hand, assumption \eqref{assumption_lemswitch} implies
\be\label{assumption_implies}
t_N^{-1}({\mathcal N}^{r\prime}_1(x,t_N) + {\mathcal N}^{r}_2(x,0)) \to 0\quad
\mbox{in probability}
\ee
The proof is concluded by applying Doob's inequality to the martingale 
\eqref{martingale_marks}, noting that the bound \eqref{bound_variation} is a uniform $O(t)$. 
\end{proof}
\mbox{}\\ \\
We now detail the behaviour of the flux over time interval $(N(t - \delta ), Nt)$. 
\begin{lemma}
\label{lemflu}
Let the sequence $(x_N) $ be as above, and $x\in\Z$. Let $y_N:=x_N+x$. Then
$$  
\limsup_{N\to+\infty}\frac{\Exp[ \Gamma^\alpha_{y_N}(Nt,\eta^N_0)
 - \Gamma^\alpha_{y_N}(N(t-\delta),\eta^N_0)]}{N\delta} 
\leq (p-q)c$$
\end{lemma}
\begin{proof}{Lemma}{lemflu}
 First we note that, by definition \eqref{current_harris},
\be\label{rewrite_diff}
\Gamma^\alpha_{y_N}(Nt,\eta^N_0)
 - \Gamma^\alpha_{y_N}(N(t-\delta),\eta^N_0)=\Gamma^\alpha_{y_N}(N(t-\delta),Nt,\eta^N_{N(t-\delta)})
\ee
By Lemma \ref{lemma_ergodic}, there exists a sequence $(z_N)_{N\in\N}$ such that 
\be\label{properties_zn}\lim_{N\to+\infty}N^{-1}z_N=0, \quad 
\lim_{N\to+\infty}\alpha(y_N-z_N)=c
\ee
By \eqref{difference_currents},
\begin{eqnarray}
& N^{-1}\Gamma^\alpha_{y_N}(N(t-\delta),Nt,\eta^N_{N(t-\delta)})-
N^{-1}\Gamma^\alpha_{y_N-z_N}(N(t-\delta),Nt,\eta^N_{N(t-\delta)}) & \nonumber\\
&
\dsp\leq N^{-1}\sum_{z=y_N-z_N+1}^{y_N}\eta^N_{N(t-\delta)}(z) & \label{bydifference}
\end{eqnarray}
Let $\varepsilon'>0$. Using the first limit in \eqref{properties_zn}, 
the r.h.s. of \eqref{bydifference} can be bounded for $N$ large enough by
\[
N^{-1} \sum_{z=y_N-N\varepsilon'}^{y_N}\eta^N_{N(t-\delta)}(z)
\]
which, by Theorem \ref{th_hydro}, converges in probability to $\int_{u-\varepsilon'}^u\rho(t-\delta,z)dz$.
Since $\varepsilon'$ is arbitrary, it follows that
\be\label{remainder}
\lim_{N\to+\infty}N^{-1}\sum_{z=y_N-z_N+1}^{y_N}\eta^N_{N(t-\delta)}(z)=0\quad\mbox{in probability}
\ee
On the other hand, by Corollary \ref{corollary_consequence} applied to $y=y_N-z_N$,
\be\label{boundbysource}
\Gamma^\alpha_{y_N-z_N}(N(t-\delta),Nt,\eta^N_{N(t-\delta)})\leq
\Gamma^\alpha_{y_N-z_N}(N(t-\delta),Nt,\eta^{*,y_N-z_N})
\ee
By  Proposition \ref{current_source}  and the second limit in \eqref{properties_zn},
\be\label{limsource}
\lim_{N\to+\infty}
\left\{
N^{-1}
\Gamma^\alpha_{y_N-z_N}(N(t-\delta),Nt,\eta^{*,y_N-z_N})
- (p-q)c
\right\}^+=0
\ee
in probability. Identity \eqref{rewrite_diff}, 
inequalities \eqref{bydifference} and \eqref{boundbysource}, 
and limits \eqref{remainder} and \eqref{limsource} imply that
\be\label{current_thus}
\left\{
\frac{\Gamma^\alpha_{y_N}(Nt,\eta^N_0)
 - \Gamma^\alpha_{y_N}(N(t-\delta),\eta^N_0)]}{N\delta}-(p-q)c
\right\}^+
\ee
vanishes in probability. But the sequence of random variables 
\eqref{current_thus} is uniformly integrable,
as can be seen by bounding the currents with suitable Poisson processes. 
Hence, the expectation of \eqref{current_thus} also vanishes, which implies the result of the lemma.
\end{proof}
\mbox{}\\ \\
\begin{proof}{Proposition}{prop:inv_env}
\mbox{}\\ \\
{\em Proof of (i).} Recall  $(P_t^\alpha)_{t\geq 0}$ 
denotes the semigroup generated by \eqref{generator}. 
We show below that for positive $t$ and  bounded  continuous $f$,
\be\label{uniform_alpha}
P_t ^{\alpha ^N }f \stackrel{\rm uniformly}{\longrightarrow} P_t ^{\overline{\alpha}} f
\ee
Since $f$ is continuous and the semigroup $(P_t^{\overline{\alpha}})_{t\geq 0}$ is Feller, we have, as $N\to+\infty$,
\be\label{since_feller}
\int fd \nu^{N,\delta}_t \ \rightarrow \int fd \nu \mbox{ and } 
\int P_t ^{\overline \alpha } fd \nu^{N,\delta}_t \ \rightarrow 
\int P_t ^{\overline \alpha } fd \nu.
\ee
Next we write
\begin{eqnarray*}
\int P_t ^{\alpha{^N}} f d\nu^{N,\delta}_t -  \int P_t ^{\overline \alpha } f d\nu & = & 
\int P_t ^{\alpha{^N}} f d\nu^{N,\delta}_t -  \int P_t ^{\overline \alpha }f d\nu^{N,\delta}_t\\
& + & \int P_t ^{\overline \alpha }f d\nu^{N,\delta}_t - \int P_t ^{\overline \alpha }f d\nu
\end{eqnarray*}
On the above right-hand side, the first line vanishes as $N\to+\infty$ by \eqref{uniform_alpha},
and the second one by \eqref{since_feller}, hence
\be\label{feller_hence}
\int P_t ^{\alpha{^N}} f d\nu^{N,\delta}_t\longrightarrow
\int P_t ^{\overline{\alpha}} f d\nu
\ee
By the usual Cesaro bounds  (see proof of
\cite[Chapter 1, Proposition 1.8(e)]{ligbook}) 
\[
\left\vert  \int P_t ^{\alpha{^N}} f d\nu^{N,\delta}_t - \int fd \nu^{N,\delta}_t  \right\vert
\ \leq \ \frac{2t}{N}||f||_\infty,
\]
Thus, recalling \eqref{since_feller}, we obtain
\[\int P_t ^{\alpha{^N}} f d\nu^{N,\delta}_t \rightarrow \int fd \nu\]
Combined with \eqref{feller_hence}, this shows $\nu$ is an equilibrium for 
semigroup $(P_t ^{\overline \alpha } )_{t \geq 0}$ as claimed.\\ \\
{\em Proof of \eqref{uniform_alpha}.}
 Let $ \beta,\beta'\in{\bf A}$, $f:\overline{\mathbf X}\to\R$ 
be a bounded local function, and $a,b\in\Z$ such that $a\leq b$ 
and $f(\eta)$ depends only on the restriction of $\eta$ to $[a,b]$. 
We claim that
\begin{eqnarray}
   \vert P_t ^\beta f(\eta)   - P_t ^{\beta'} f(  \eta)   \vert
& \leq & 2   || f||_\infty 
[2 \Prob(\mathcal P(g(\infty )t) \geq m^ \prime )\label{tomsclaim}\\
& + & (2m^\prime +(b-a+1))
\Big\{ 1-e^{-g(\infty) \sup_{x: d(x,[a,b]) \leq m^\prime}|\beta(x) - \beta'(x)|t}\Big\}\nonumber,
\end{eqnarray}
which implies \eqref{uniform_alpha}. To prove \eqref{tomsclaim},
we write 
\be\label{jensen_lipf}
   \vert P_t ^\beta f(\eta)   - P_t ^{\beta'} f(  \eta)   \vert \leq 
	2||f||_\infty\Prob\left(
	\exists x\in[a,b]:\,\eta^\beta_t(x)\neq \eta^{\beta'}_t(x)
	\right)
\ee
where $(\eta_t^\beta)_{t\geq 0}$ and $(\eta_t^{\beta'})_{t\geq 0}$ 
denote processes with respective environments $\beta$ and $\beta'$. 
Though these environments (may) differ, we can still couple them as
in Subsection \ref{subsec:Harris} by using the same Poisson measure. 
However, unlike what happens when coupling two processes in the same 
environment (see Subsection \ref{subsec:Harris}), a pair of opposite 
discrepancies may be created. This may happen for instance if 
$\eta_{t-}^\beta(x)>\eta_{t-}^{\beta'}(x)$
but an $\eta^{\beta'}$ particle jumps alone (which may only occur 
if $\beta'(x)>\beta(x)$) to a site $x+z$ such that 
$\eta_{t-}^\beta(x+z)<\eta_{t-}^{\beta'}(x+z)$. The rate of such 
creation at site $x$ cannot exceed $|\beta(x)-\beta'(x)|$. 
Thus the probability that at least one pair of opposite discrepancies 
has been created by time $t$ by a jump from a site in $[a-m',b+m']$ 
is controlled by the second term on the r.h.s. of \eqref{tomsclaim}. 
Outside this event, any discrepancy present in $[a,b]$ at time $t$ 
has been created outside $[a-m',b+m']$. Since
all discrepancies, once created, jump with maximum rate $1=g(+\infty)$, 
the rightmost discrepancy created to the left of $a-m'$ and the leftmost 
discrepancy created to the right of $b+m'$ can be controlled by rate 
$1$ Poisson processes. Thus the probability of such a discrepancy 
having reached $[a,b]$ by time $t$ is controlled
 by the first term on the r.h.s. of \eqref{tomsclaim}.\\ \\
{\em Proof of (ii).}  This follows from statement \eqref{eq:loc_eq_super_bis} of Theorem \ref{th_strong_loc_eq_super}. \\ \\
{\em Proof of (iii).}
 We apply Lemma \ref{lemswitch} to $\kappa^N:=\alpha^N:=\tau_{x_N}\alpha$ and $\zeta^N_t:=\tau_{x_N}\eta^{\alpha,N}_t$, taking successively $t_N=Nt$ and $t_N=N(t-\delta)$. Assumption \eqref{assumption_lemswitch} of the Lemma is satisfied, because for any $s\geq 0$, $\varepsilon>0$ and $x\in\Z$,
\[
N^{-1}\zeta^N_{Ns}(x) \leq N^{-1}\sum_{y=x}^{x+N\varepsilon}\zeta^N_{Ns}(y)=
N^{-1}\sum_{y=x}^{x+N\varepsilon}\eta^N_{Ns}(y+x_N)
\]
and the latter converges to $\int_0^\varepsilon\rho(s,u+z)dz$ in probability (for $s=0$, this follows from the assumption of an initial density profile in Theorem \ref{th_hydro}. For $s>0$, this follows from the conclusion of the theorem).
By difference, using definition \eqref{def_nutdelta} of $\nu_t^{N,\delta}$ to rewrite the time integral in Lemma \ref{lemswitch}, we obtain
\[
\Exp\left\{
\frac{\Gamma_x^{\alpha^N}(Nt,\zeta^N_0)-\Gamma_x^{\alpha^N}(N(t-\delta),\zeta^N_0)}{N\delta}\right\}-\alpha^N(x)(p-q)
\int_{\overline{\mathbf X}}g(\eta(x))d\nu_t^{N,\delta}(d\eta)
\]
vanishes as $N\to+\infty$.
We apply  Lemma \ref{lemflu} to the above expectation, noting that, for any $s\geq 0$,
\[
\Gamma_x^{\alpha^N}(s,\zeta^N_0)=\Gamma_{x_N+x}^\alpha(s,\eta^N_0)
\]
Thus,  using the convergences $\alpha^N(x)\to\overline{\alpha}(x)$ and $\nu_t ^{N,\delta}\to\nu$, as well as (ii), we obtain the desired conclusion.
\end{proof}
\section{Proof of Theorem \ref{cor_dream_bfl}}\label{sec:proof_bfl}
{\em Step one. }We first prove the result in the particular case where $\eta_0(x)=0$ for all $x\geq 0$.
We need the following result.
\begin{proposition}\label{prop_riemann}
Let $\rho\geq 0$, and let $\mathcal R_{\rho,0}(.,.)$ denote the entropy solution to \eqref{conservation_law} with initial condition
\be\label{initial_riemann}
\mathcal R_{\rho,0}(0,x)=\rho{\bf 1}_{(-\infty,0)}(x)
\ee
Then for every $t>0$, 
\be\label{limits_riemann}
\liminf_{(t',u')\to(t,0)}\mathcal R_{\rho,0}(t',u')=\rho\wedge\rho_c, \quad\limsup_{(t',u')\to(t,0)}\mathcal R_{\rho,0}(t',u')\color[rgb]{0,0,0}\color[rgb]{0,0,0}=\rho
\ee
In particular, if $\rho\leq\rho_c$, ${\mathcal R}_{\rho,0}$ is continuous at $(t,0)$.
\end{proposition}
\begin{proof}{Proposition}{prop_riemann}
By \cite[Proposition 4.3]{bgrs5}, the entropy solution satisfies
\begin{eqnarray}\label{entropy_riemann_1}
\limsup_{(t',u')\to(t,u)}{\mathcal R}_{\rho,0}(t',u') & = & \sup{\rm argmax}_{r\in[0,\rho]}\left[f(r)-\frac{u}{t}r\right]\\
\liminf_{(t',u')\to(t,u)}{\mathcal R}_{\rho,0}(t',u') & = & \inf{\rm argmax}_{r\in[0,\rho]}\left[f(r)-\frac{u}{t}r\right]\label{entropy_riemann_2}
\end{eqnarray}
where ${\rm argmax}$ denotes the set of maximizers. We apply this to $u=0$.
Since $f$ is strictly increasing on $[0,\rho_c]$, if $\rho\leq\rho_c$, the set of maximizers in \eqref{entropy_riemann_1}--\eqref{entropy_riemann_2} reduces to $\rho$.
If $\rho>\rho_c$, this set is $[\rho,\rho_c]$.
\end{proof}
\mbox{}\\ \\
The convergence in Theorem \ref{cor_dream_bfl} restricted to integer times
then follows from Theorems \ref{th_strong_loc_eq} and \ref{th_strong_loc_eq_super} by taking 
$\eta^N_0=\eta_0$, $x_N=0$ for all $N\in\N$, and using Proposition \ref{prop_riemann} with $u=0$ and $t=1$.
Indeed, the above definition of the sequence $(\eta^N_0)_{N\in\N}$ implies that this sequence has density
profile $\mathcal R_{\rho,0}(0,.)$. Besides, the particular sequence $x_N=0$ is typical in the sense of Definition \ref{def_typical}.\\ \\
To fill the gap between integer times and all times we use the semigroup property to write
the law of $\eta_t^\alpha$ as follows:
\[
\delta_{\eta_0}P^\alpha_t=\left(\delta_{\eta_0}P^\alpha_{\lfloor t \rfloor}\right)P^\alpha_{t-\lfloor t\rfloor}
\]
Using the limit established for integer times, we have 
\[
\left(\delta_{\eta_0}P^\alpha_{\lfloor t \rfloor}\right)\rightarrow\mu^{\alpha,\rho\wedge\rho_c}
\]
Since $\mu^{\alpha,\rho\wedge\rho_c}$ is invariant for $(P^\alpha_s)_{s\geq 0}$ and the semigroup is weakly continuous, it follows that $\delta_{\eta_0}P^\alpha_t\to\mu^{\alpha,\rho\wedge\rho_c}$.\\ \\
{\em Step two. } We now consider the general case. We couple the process $(\eta^{\alpha}_t)_{t\geq 0}$ with the process $(\zeta^{\alpha}_t)_{t\geq 0}$ whose initial state is defined by
\[
\zeta_0(x):=\eta_0(x){\bf 1}_{(-\infty,0)}(x)
\]
Since $\zeta^N_0\leq\eta^N_0$ and the mapping \eqref{mapping} is nondecreasing, we have 
$\zeta^{\alpha}_t\leq\eta^{\alpha}_t$ for all $t\geq 0$. Let
$\gamma^{\alpha}_t:=\eta^{\alpha}_t-\zeta^{\alpha}_t$ be the configuration of 
$\eta$ particles in excess with respect to $\zeta$ particles.
The result will follow if we show that, for any $x\in\Z$, 
\be\label{no_more_second}
\lim_{t\to+\infty}\Prob\left(
\gamma^{\alpha}_t(x)=0
\right)=1
\ee
Indeed, let $\psi$ be a bounded local function on $\mathbf X$, 
depending only on sites $x\in[-R,R]$ for some $R\in\N$, and let $E_t$ 
denote the event that $\gamma_t^\alpha(x)=0$ for all $x\in[-R,R]$. 
Since $\psi(\eta_t^\alpha)=\psi(\zeta_t^\alpha)$ on $E_t$, we have
$$
\left|
\Exp\psi(\eta_t^\alpha)-\Exp\psi(\zeta_t^\alpha)
\right|\leq 2||\psi||_\infty\Prob(E_t^c)
$$
which vanishes by \eqref{no_more_second}.
To prove \eqref{no_more_second}, using the definition and dynamics of classes 
from Subsection \ref{sec:class}, we intepret $\zeta$ particles as first class particles and 
$\gamma$ particles as second class particles among $\eta$ particles. In particular, 
we label $\gamma$ particles increasingly from left to right, $0$ being the label of the leftmost particle.
The dynamics of second class particles was defined in such a way that the order of labels is maintained during the evolution, and the skeleton of the trajectory of each second class particle is a random walk with jump probability $p$ (resp. $q$) to the right (resp. left).  Since $p>q$, any $\gamma$ particle that jumps infinitely many times goes to $+\infty$. Let us denote by $N_i\in\N\cup\{+\infty\}$ the number of jumps performed by the $\gamma$ particle with label $i$, and by $X_i\in\Z\cup\{+\infty\}$ the final location of this particle. Note that $X_i=+\infty$ if $N_i=+\infty$, and $X_i<+\infty$ if $N_i<+\infty$.
In order to prove \eqref{no_more_second}, it is enough to prove that 
\be\label{no_more_second_enough}
X_0=+\infty,\,\mbox{a.s.}
\ee
Indeed, since $\gamma$ particles remain ordered, \eqref{no_more_second_enough} implies that with probabilty one, we have
$X_i=+\infty$ for all $i\in\N$, thus for all $x\in\Z$, $\gamma_t^\alpha(x)$ converges a.s. to $0$ as $t\to +\infty$. This implies convergence in law, and thus \eqref{no_more_second}.
We now prove \eqref{no_more_second_enough}. Let $x\in\Z$.
Since $\eta_t^\alpha\geq\gamma_t^\alpha$, on the event $\{X_i=x\}$,  there is a (random) time $T$ such that
$\eta_t^\alpha(x)\geq 1$ for all $t\geq T$. 
Hence, $X_i=+\infty$ is a consequence of the following result.
\begin{lemma}\label{lemma_empty}
For every $x\in\Z$, 
\be\label{liminf_empty}
\liminf_{t\to+\infty}\eta^\alpha_t(x)=0
\ee
\end{lemma}
\begin{proof}{Lemma}{lemma_empty}
We couple the process $(\eta^\alpha_t)_{t\geq 0}$ to the critical stationary process 
$(\xi^{\alpha,\rho_c}_t)_{t\geq 0}$ and to the process $(\xi^{\alpha}_t)_{t\geq 0}$ 
with generator \eqref{generator} and initial state
$\xi_0:=\max\left(\eta_0,\xi_0^{\alpha,\rho_c}\right)$. Since the mapping 
\eqref{mapping} is nondecreasing, we have 
\be\label{wehave_empty}
\max\left(\eta^{\alpha}_t,\xi^{\alpha,\rho_c}_t\right)\leq\xi^\alpha_t
\ee
for all $t\geq 0$. By \cite[Theorem 2.1]{bmrs1}, $\xi^\alpha_t$ converges in law 
as $t\to+\infty$ to $\mu^{\alpha,\rho_c}$, that is the law of $\xi^{\alpha,\rho_c}_t$. 
We thus have
\begin{eqnarray*}
\liminf_{t\to+\infty}\Prob\left(\eta_t^\alpha(x)=0\right) & \geq &
\liminf_{t\to+\infty}\Prob\left(\xi_t^\alpha(x)=0\right)\\
& = & 
\Prob\left(\xi_t^{\alpha,\rho_c}(x)=0\right)=\mu^{\alpha,\rho_c}(\eta(x)=0)>0
\end{eqnarray*}
The result follows.
We now conclude the proof by contradiction. Assuming \eqref{liminf_empty} does not hold, 
there exists an a.s. finite random time $T$ such that $\eta^\alpha_t(x)>0$ for all $t\geq T$. 
But
\[
\Prob(\eta^\alpha_t=0)=\Prob(\eta^\alpha_t=0,t\geq T)+\Prob(\eta^\alpha_t=0,t< T)
\]
The second probability on the r.h.s. vanishes as $t\to+\infty$ because $T$ is finite, 
while the first one is zero, because the event inside it is impossible by definition of $T$.
\end{proof}
\mbox{}\\ \\
\noindent {\bf Acknowledgements:}
 We thank Gunter Sch\"utz 
for many interesting discussions.
This work was partially supported by laboratoire MAP5,
grants ANR-15-CE40-0020-02 and ANR-14-CE25-0011, 
Simons Foundation Collaboration grant 281207 awarded to K. Ravishankar.
We thank 
Universit\'{e}s Clermont Auvergne and Paris Descartes for hospitality.
This work was partially carried out during C.B.'s 2017-2018 
d\'el\'egation CNRS, whose support is acknowledged.
Part of it was done during the authors' stay at the 
Institut Henri Poincar\'e - Centre Emile Borel for
the trimester ``Stochastic Dynamics Out of Equilibrium'', and during the authors' stay at NYU Shanghai.
The authors thank these institutions for hospitality and support. 
\end{document}